\documentclass[a4paper,reqno,final]{amsproc}%
\usepackage{amsmath}
\usepackage{amssymb}
\usepackage{amsfonts}%
\usepackage{graphicx}
\providecommand{\U}[1]{\protect\rule{.1in}{.1in}}
\theoremstyle{plain}

\newtheorem{definition}{Definition}

\newtheorem{lemma}{Lemma}

\newtheorem{proposition}{Proposition}
\newtheorem{remark}{Remark}

\newtheorem{theorem}{Theorem}
\numberwithin{equation}{section}
\begin{document}
\title{Proof of the BMV Conjecture}
\author{Herbert R Stahl}
\curraddr{Beuth Hochschule/FB II; Luxemburger Str. 10; 13 353 Berlin; Germany}
\email{HerbertRStahl@Googlemail.com}
\thanks{The research has been supported by the grant STA 299/13-1 der Deutschen
Forschungsgemeinschaft (DFG)}
\subjclass[2000]{ Primary 15A15, 15A16; Secondary 30F10, 44A10}
\keywords{BMV conjecture, Laplace transform, special matrix functions}
\date{August 8, 2012}

\begin{abstract}
We prove the BMV (\underline{B}essis, \underline{M}oussa, \underline{V}illani,
\cite{BMV75}) conjecture, which states that the function $t\mapsto
\operatorname*{Tr}\exp(A-tB)$, $t\geq0,$ is the Laplace transform of a
positive measure on $[0,\infty)$ if $A$ and $B$ are $n\times n$ Hermitian
matrices and $B$ is positive semidefinite. A semi-explicit representation for
this measure is given.

\end{abstract}
\maketitle

\section{Introduction}

\subsection{\label{sec_1_1}The Conjecture}

\qquad Let $A$ and $B$ be two $n\times n$ Hermitian matrices and let $B$ be
positive semidefinite. In \cite{BMV75} it has been conjectured that under
these assumptions the function%
\begin{equation}
f(t):=\operatorname*{Tr}e^{A-tB},\text{ \ \ }t\geq0,\label{f1_1a0}%
\end{equation}
can be represented as the Laplace transform%
\begin{equation}
f(t)=\int e^{-t\,s}d\mu_{A,B}(s)\label{f1_1a}%
\end{equation}
of a positive measure $\mu_{A,B}$ on $\mathbb{R}_{+}=[0,\infty)$. In the
present article we prove this conjecture from 1975 and give a semi-explicit
expression for the measure $\mu_{A,B}$ (cf. Theorems \ref{sec1_thm1} and
\ref{sec1_thm2}, below).

Over the years different approaches and techniques have been tested for
proving the conjecture. Surveys are contained in \cite{Moussa00} and
\cite{Hillar07}. Recent publications are typically concerned with techniques
from non-commutative algebra and combinatorics (\cite{HillarJohnson05},
\cite{JohnsonLeichenauerMcNamaraCostas05}, \cite{Hansen06}, \cite{Haegele07},
\cite{Hillar07}, \cite{KlepSchweighofer08}, \cite{LandweberSpeer09},
\cite{CollinsDykemaTorres10}, \cite{FleischhackFriedland10}, \cite{Burgdorf11}%
). This direction of research was opened by a reformulation of the problem in
\cite{LiebSeiringer04}. Although our approach will follow a different line
of\ analysis, we nevertheless repeat the main assertions from
\cite{LiebSeiringer04} in the next subsection as points of reference for later discussions.

\subsection{\label{sec1_2}Reformulations of the Conjecture}

\begin{definition}
\label{sec1_def1}A function $f\in C^{\infty}(\mathbb{R}_{+})$ is called
\textbf{completely monotonic} if%
\[
(-1)^{m}f^{(m)}(t)\geq0\text{ \ \ for all \ \ }m\in\mathbb{N}\text{ \ and
\ }t\in\mathbb{R}_{+}.
\]

\end{definition}

By Bernstein's theorem about completely monotonic functions (cf.
\cite{Donoghue} or \cite[Chapter IV]{Widder}) this property is equivalent to
the existence of the Laplace transform (\ref{f1_1a}) with a positive measure
on $\mathbb{R}_{+}$. In this way, Definition \ref{sec1_def1} gives a first
reformulation of the BMV conjecture.\smallskip

In \cite{LiebSeiringer04} two other reformulations have been proved. It has
been shown that the conjecture is equivalent to each of the following two assertions:

\begin{enumerate}
\item[(i)] Let $A$ and $B$ be two positive semidefinite Hermitian matrices.
For each $m\in\mathbb{N}$ the polynomial $t\mapsto Tr(A+tB)^{m}$ has only
non-negative coefficients.

\item[(ii)] Let $A$ be a positive definite and $B$ a positive semidefinite
Hermitian matrix. For each $p>0$ the function $t\mapsto Tr(A+tB)^{-p}$ is the
Laplace transform of a positive measure on $\mathbb{R}_{+}$.
\end{enumerate}

Especially, reformulation (i) has paved the way for extensive research
activities with tools from non-commutative algebra; several of the papers have
been mentioned earlier. The parameter $m$ in assertion (i) introduces a new
and discrete gradation of the problem. Presently, assertion (i) has been
proved for $m\leq13$ (cf. \cite{Haegele07}, \cite{KlepSchweighofer08}). The
BMV-conjecture itself is still unproven, even for the general case of matrices
with a dimension as low as $n=3$. In the diploma thesis \cite{Grafendorfer07}
the case $n=3$ has been investigated very carefully by a combination of
numerical and analytical tools, but no counterexample could be found.

In \cite{LiebSeiringer04} one also finds a short review of the relevance of
the BMV conjecture in mathematical physics, the area from which the problem
arose originally.\footnote{Meanwhile, in a follow-up paper
\cite{LiebSeiringer12} to \cite{LiebSeiringer04}, the reformulations of the
BMV conjecture have been extended, and the conjecture itself has been
generalised by replacing the expression on the left-hand side of
(\ref{f1_1a0}) by elementary symmetric polynomials of order $m\in
\{1,\ldots,n\}$ of exponentials of the $n$ eigenvalues of the expression
$A-t\,B$. The expression in (\ref{f1_1a0}) with the trace operator then
corresponds to the case $m=1$.}\medskip

Among the earlier investigations of the conjecture, especially
\cite{MehtaKumar76} has been very impressive and fascinating for the author.
There, already in 1976, the conjecture was proved for a rather broad class of
matrices, including the two groups of examples with explicit solutions that we
will state next.

\subsection{Two Groups of Examples with Explicit Solutions\smallskip}

\subsubsection{\label{sec1_3_1}Commuting Matrices $A$ and $B$}

\qquad If the two matrices $A$ and $B$ commute, then they can be diagonalized
simultaneously, and\ consequently the BMV conjecture becomes solvable rather
easily; the measure $\mu_{A,B}$ in (\ref{f1_1a})\ is then given by%
\begin{equation}
\mu_{A,B}=\sum_{j=1}^{n}e^{a_{j}}\,\delta_{b_{j}}\label{f1_1b}%
\end{equation}
with $a_{1},\ldots,a_{n}$ and $b_{1},\ldots,b_{n}$ the eigenvalues of the two
matrices $A\ $and $B,$ respectively, and $\delta_{x}$ the Dirac measure at the
point $x$. Indeed, the trace of a matrix $M$ is invariant under similarity
transformations $M\mapsto T\,M\,T^{-1}$. Therefore, we can assume without loss
of generality that $A$ and $B$ are given in diagonal form, and measure
(\ref{f1_1b}) follows immediately.\smallskip

\subsubsection{Matrices of Dimension $n=2$}

\qquad We consider $2\times2$ Hermitian matrices $A$ and $B$ with $B$ assumed
to be positive semidefinite. In order to keep notations simple, we assume $B$
to be given in diagonal form $B=\operatorname*{diag}(b_{1},b_{2})$ with $0\leq
b_{1}\leq b_{2}$.

If $b_{1}=b_{2}$, then, without loss of generality, also the matrix $A$ can be
assumed to be given in diagonal form, and consequently the case is covered by
(\ref{f1_1b}). Thus, we have to consider only the situation that%
\begin{equation}
A=\left(
\begin{array}
[c]{cc}%
a_{11} & a_{12}\\
\overline{a}_{12} & a_{22}%
\end{array}
\right)  ,\text{ \ \ }B=\left(
\begin{array}
[c]{cc}%
b_{1} & 0\\
0 & b_{2}%
\end{array}
\right)  ,\text{ \ \ }0\leq b_{1}<b_{2}<\infty.\label{f1_1c}%
\end{equation}

\begin{proposition}
\label{sec1_prop1}If the matrices $A$ and $B$ are given by (\ref{f1_1c}), then
the function $t\mapsto\operatorname*{Tr}\exp(A-tB)$, $t\in\mathbb{R}_{+}$, in
(\ref{f1_1a0}) can be represented as a Laplace transform (\ref{f1_1b}) with
the positive measure%
\begin{equation}
d\mu_{A,B}(t)=e^{a_{11}}d\delta_{b_{1}}(t)+e^{a_{22}}d\delta_{b_{2}%
}(t)+w_{A,B}(t)\chi_{\left(  b_{1},b_{2}\right)  }(t)dt,\text{ \ \ }%
t\in\mathbb{R}_{+},\label{f1_1d}%
\end{equation}
where $\chi_{\left(  b_{1},b_{2}\right)  }$\ denotes the characteristic
function of the interval $\left(  b_{1},b_{2}\right)  ,$ and the density
function $w_{A,B}$ is given by%
\begin{align}
& w_{A,B}(t)=\frac{4}{(b_{2}-b_{1})\,\pi}\exp\left(  \frac{a_{11}%
(b_{2}-t)+a_{22}(t-b_{1})}{b_{2}-b_{1}}\right)  \times\medskip\label{f1_1e}\\
& \text{ \ \ \ \ \ \ \ \ \ \ \ \ \ \ \ \ \ \ \ \ \ \ }\times\int_{0}%
^{|a_{12}|}\cos\left(  \frac{b_{2}+b_{1}-2\,t}{b_{2}-b_{1}}\,u\right)
\,\sinh\left(  \sqrt{|a_{12}|^{2}-u^{2}}\right)  du.\ \ \nonumber
\end{align}
This density function is positive for all $b_{1}<t<b_{2}$.\smallskip
\end{proposition}

Proposition \ref{sec1_prop1} will be proved in Section \ref{sec6}. In
\cite{MehtaKumar76} an explicit solution has also been proved for dimension
$n=2$; there the density function looks rather different from (\ref{f1_1e}),
and it has the advantage that its positivity can be recognized immediately,
while in our case of (\ref{f1_1e}) a nontrivial proof of positivity is
required (cf. Subsection \ref{sec6_2}).\medskip

\subsection{\label{sec1_4}The Main Result}

\qquad We prove two theorems. In the first one it is just stated that the BMV
conjecture is true, while in the second one we give a semi-explicit
representation for the positive measure $\mu_{A,B}$ in the Laplace transform
(\ref{f1_1a}). In many respects this second theorem is a generalization of
Proposition \ref{sec1_prop1}.

\begin{theorem}
\label{sec1_thm1}If $A$ and $B$ are two Hermitian matrices with $B$ positive
semidefinite, then there exists a unique positive measure $\mu_{A,B}$ on
$[0,\infty)$ such that (\ref{f1_1b}) holds for $t\geq0$. In other words: the
BMV conjecture holds true.
\end{theorem}

For the formulation of the second theorem we need some preparations.

\begin{lemma}
\label{sec1_lem1}Let $A$ and $B$ be the two matrices from Theorem
\ref{sec1_thm1}. Then there exists a unitary matrix $T_{0}$ such that the
transformed matrices $\widetilde{A}=(\widetilde{a}_{ij}):=T_{0}^{\ast}AT_{0} $
and $\widetilde{B}:=T_{0}^{\ast}BT_{0}$ satisfy%
\begin{equation}
\widetilde{B}=\operatorname*{diag}\left(  \widetilde{b}_{1},\ldots
,\widetilde{b}_{n}\right)  \text{ \ \ with \ \ }0\leq\widetilde{b}_{1}%
\leq\cdots\leq\widetilde{b}_{n},\label{f1_2b}%
\end{equation}
and%
\begin{equation}
\widetilde{a}_{ij}=0\text{ \ \ for all\ \ \ }i,j=1,\ldots,n,i\neq j\text{
\ with \ }\widetilde{b}_{i}=\widetilde{b}_{j}.\label{f1_2c}%
\end{equation}

\end{lemma}

\begin{proof}
The existence of a unitary matrix $T_{0}$ such that (\ref{f1_2b}) holds is
guaranteed by the assumption that $B$ is Hermitian and positive semidefinite.
If all $\widetilde{b}_{j}$ are pairwise different, then requirement
(\ref{f1_2c}) is void. If however several $\widetilde{b}_{j}$ are identical,
then one can rotate the corresponding subspaces in such a way that in addition
to (\ref{f1_2b}) also (\ref{f1_2c}) is satisfied.\smallskip
\end{proof}

Since the matrix $A-tB$ is Hermitian for $t\in\mathbb{R}_{+}$, there exists a
unitary matrix $T_{1}=T_{1}(t)$ such that%
\begin{equation}
T_{1}^{\ast}(A-t\,B)T_{1}=\operatorname*{diag}\left(  \lambda_{1}%
(t),\ldots,\lambda_{n}(t)\right)  .\label{f1_2d}%
\end{equation}
The $n$ functions $\lambda_{1},\ldots,\lambda_{n}$ in (\ref{f1_2d}) are
restrictions to $\mathbb{R}_{+}$ of branches of the solution $\lambda$ of the
polynomial equation%
\begin{equation}
g(\lambda,t):=\det\left(  \lambda\,I-(A-t\,B)\right)  =0,\label{f1_2e}%
\end{equation}
i.e., $\lambda_{j}$, $j=1,\ldots,n$, is a branch of the solution $\lambda$ if
the pair $(\lambda,t)=(\lambda_{j}(t),t)$ satisfies (\ref{f1_2e}) for each
$t\in\mathbb{C}$. The solution $\lambda$ is an algebraic function of degree
$n$ if the polynomial $g(\lambda,t)$ is irreducible, and it consists of
several algebraic functions otherwise. In the most extreme situation, the
polynomial $g(\lambda,t)$ can be factorized into $n$ linear factors, and this
is exactly the case when the two matrices $A$ and $B$ commute, which has been
discussed in Subsection \ref{sec1_3_1}.

In any case, the solution $\lambda$ of (\ref{f1_2e}) consists of one or
several multivalued functions of $t$ in $\mathbb{C}$, and the total number of
different branches $\lambda_{j},$ $j=1,\ldots,n$, is always exactly $n$. In
the next lemma, properties of the functions $\lambda_{j},$ $j=1,\ldots,n$, are
assembled, which are relevant for the formulation of Theorem \ref{sec1_thm2}.
The lemma will be proved in a slightly reformulated form as Lemma
\ref{sec2_Lem1} in Section \ref{sec2}.\smallskip

\begin{lemma}
\label{sec1_lem2}There exist $n$ different branches $\lambda_{j},$
$j=1,\ldots,n$, of the solution $\lambda$ of (\ref{f1_2e}). Each one can be
assumed to be analytic in a punctured neighborhood of infinity, none of them
has a branch point at infinity, and they can be numbered in such a way that we
have%
\begin{equation}
\lambda_{j}(t)\,=\,\widetilde{a}_{jj}-\widetilde{b}_{j}t+\text{O}\left(
1/t\right)  \text{ \ as \ }t\rightarrow\infty,\text{ \ }j=1,\ldots
,n,\label{f1_2f}%
\end{equation}
where the coefficients $\widetilde{a}_{jj},\widetilde{b}_{j}$, $j=1,\ldots,n
$, are elements of the matrices $\widetilde{A}$ and $\widetilde{B}%
$\ introduced in Lemma \ref{sec1_lem1}.\smallskip
\end{lemma}

With Lemmas \ref{sec1_lem1} and \ref{sec1_lem2} we are ready to formulate the
second theorem.\smallskip

\begin{theorem}
\label{sec1_thm2}For the measure $\mu_{A,B}$ in (\ref{f1_1b}) we have the
representation%
\begin{equation}
d\mu_{A,B}(t)\,=\,\sum_{j=1}^{n}e^{\widetilde{a}_{jj}}d\delta_{\widetilde
{b}_{j}}(t)+w_{A,B}(t)dt,\text{ \ \ }t\in\mathbb{R}_{+},\label{f1_3a}%
\end{equation}
with a density function $w_{A,B}$ that can be represented as%
\begin{equation}
w_{A,B}(t)\,=\,\sum_{\widetilde{b}_{j}<t}\frac{1}{2\pi i}\oint\nolimits_{C_{j}%
}e^{\lambda_{j}(\zeta)+t\,\zeta}d\zeta,\text{ \ \ for \ \ }t\in\mathbb{R}%
_{+},\label{f1_3b}%
\end{equation}
or equivalently as%
\begin{equation}
w_{A,B}(t)\,=\,-\,\sum_{\widetilde{b}_{j}>t}\frac{1}{2\pi i}\oint
\nolimits_{C_{j}}e^{\lambda_{j}(\zeta)+t\,\zeta}d\zeta,\text{ \ \ for
\ \ }t\in\mathbb{R}_{+},\label{f1_3c}%
\end{equation}
where each integration path $C_{j}$ is a positively oriented, rectifiable
Jordan curve in $\mathbb{C}$ with the property that the corresponding function
$\lambda_{j}$ is analytic on and outside of $C_{j}$. The values $\widetilde
{a}_{jj}$, $\widetilde{b}_{j}$, $j=1,\ldots,n$, have been introduced in Lemma
\ref{sec1_lem1}, and the functions $\lambda_{j}$, $j=1,\ldots,n$, in Lemma
\ref{sec1_lem2}.

The measure $\mu_{A,B}$ is positive, its support satisfies%
\begin{equation}
\operatorname*{supp}(\mu_{A,B})\subseteq\lbrack\widetilde{b}_{1},\widetilde
{b}_{n}],\smallskip\label{f1_3d}%
\end{equation}
and the density function $w_{A,B}$ is a restriction of an entire function in
each interval of $[\widetilde{b}_{1},\widetilde{b}_{n}]\diagdown
\{\widetilde{b}_{1},\ldots,\widetilde{b}_{n}\}$.
\end{theorem}

Obviously, the non-negativity of the density function $w_{A,B}$ is, prima
vista, not evident from representation (\ref{f1_3b}) or (\ref{f1_3c}); its
proof will be the topic of Section \ref{sec5}.

The semi-explicit representation of the measure $\mu_{A,B}$ in Theorem
\ref{sec1_thm2} is of key importance for our strategy for a proof of the BMV
conjecture, but it probably possesses also independent value. In any case, it
already conveys some ideas about the nature of the solution.

\subsection{Outline of the Paper}

\qquad Theorem \ref{sec1_thm1} is practically a corollary of Theorem
\ref{sec1_thm2}, and the proof of Theorem \ref{sec1_thm2} is given\ in
Sections \ref{sec1b} through \ref{sec5b}.

We start in Section \ref{sec1b} with two technical assumptions, which simplify
the notation, but do not restrict the generality of the treatment. After that
in Section \ref{sec2} we compile and prove results concerning the solution
$\lambda$ of (\ref{f1_2e}) and the associated complex manifold $\mathcal{R}%
_{\lambda}$, which is the natural domain of definition for $\lambda$.

In Section \ref{sec3} all assertions in Theorem \ref{sec1_thm2} are proved,
except for the positivity of the measure $\mu_{A,B}$.

The proof of positivity of $\mu_{A,B}$ follows then in Section \ref{sec5}, and
everything concerning the proofs of the Theorems \ref{sec1_thm1} and
\ref{sec1_thm2} is summed up in Section \ref{sec5b}.

The proof of Proposition \ref{sec1_prop1} follows in Section \ref{sec6}.

\section{\label{sec1b}Technical Assumptions}

\noindent\textbf{Assumption 1}. \textit{Throughout Sections \ref{sec2} through
\ref{sec5b} we assume the matrices }$A$\textit{\ and }$B$\ to be given in the
form (\ref{f1_2b}) and (\ref{f1_2c}) of Lemma \ref{sec1_lem1}\textit{, i.e.,
we have}%
\begin{equation}
B=\operatorname*{diag}\left(  b_{1},\ldots,b_{n}\right)  \text{ \ with
\ }0\leq b_{1}\leq\cdots\leq b_{n}<\infty,\text{ \ and}\label{f1b_1a}%
\end{equation}%
\begin{equation}
a_{ij}=0\text{ \ \ for all\ \ \ }i,j=1,\ldots,n,i\neq j\text{ \ with \ }%
b_{i}=b_{j}.\label{f1b_1b}%
\end{equation}

\noindent\textbf{Assumption 2}. \textit{Further, we assume that}%
\begin{equation}
0\,<\,b_{1}\,\leq\,\cdots\,\leq\,b_{n},\label{f1b_2a}%
\end{equation}
\textit{i.e., the matrix }$B$\textit{\ is assumed to be positive
definite.}$\medskip$

Assumption 1 has the advantage that in the sequel we can write $a_{ij}$ and
$b_{j}$ instead of $\widetilde{a}_{ij}$ and $\widetilde{b}_{j}$,
$j=1,\ldots,n$.$\smallskip$

\begin{lemma}
\label{sec1b_lem1}The Assumptions 1 and 2 do not restrict the generality of
the proof of Theorems \ref{sec1_thm1} and \ref{sec1_thm2}.$\medskip$
\end{lemma}

\begin{proof}
In Lemma \ref{sec1_lem1} it has been shown that there exists a similarity
transformation $M\mapsto T_{0}^{\ast}MT_{0}$ with $T_{0}$ a unitary matrix
such that any admissible pair of matrices $A$ and $B$ is transformed\ into
matrices $\widetilde{A}$ and $\widetilde{B}$ that have the special form of
(\ref{f1b_1a}) and (\ref{f1b_1b}). Since the trace of a matrix is invariant
under such similarity transformations, we have%
\[
f(t)=\operatorname*{Tr}e^{A-tB}=\operatorname*{Tr}T_{0}^{\ast}e^{A-tB}%
T_{0}=\operatorname*{Tr}e^{T_{0}^{\ast}AT_{0}-t\,T_{0}^{\ast}BT_{0}}%
\]
for all $t\in\mathbb{R}_{+}$, which shows that the function $f$ in
(\ref{f1_1a0}) remains invariant, and consequently the generality of the
proofs of Theorems \ref{sec1_thm1} and \ref{sec1_thm2} is not restricted by
Assumption 1.

If (\ref{f1b_2a}) is not satisfied, then the matrix $\widetilde{B}%
:=B+\varepsilon I=\operatorname*{diag}\left(  \widetilde{b}_{1},\ldots
,\widetilde{b}_{n}\right)  $ with $\varepsilon>0$ satisfies Assumption 2. We
have $\widetilde{b}_{j}=b_{j}+\varepsilon$, $j=1,\ldots,n$, and it follows
from (\ref{f1_1a0}) that%
\begin{equation}
\widetilde{f}(t):=\operatorname*{Tr}e^{A-t\widetilde{B}}=e^{-\varepsilon
t}\operatorname*{Tr}e^{A-tB}=e^{-\varepsilon t}f(t)\text{ \ for \ }%
t\geq0.\label{f1b_3a}%
\end{equation}

From (\ref{f1b_3a}) and the translation property of Laplace transforms, we
deduce that the measure $\mu_{A,B}$ in (\ref{f1_1a}) for the function $f$ is
the image of the measure $\mu_{A,\widetilde{B}}$ for the function
$\widetilde{f}$\ under the translation $t\mapsto t-\varepsilon$. Consequently,
the proofs of Theorems \ref{sec1_thm1} and \ref{sec1_thm2} for the matrices
$A$ and $\widetilde{B}$ carries over to the situation with the original
matrices $A$ and $B.$\
\end{proof}

\section{\label{sec2}Preparatory Results}

\qquad In the present section we compile some results and definitions that are
concerned with the solution $\lambda$ of the polynomial equation
(\ref{f1_2e}), and in addition we introduce a complex manifold $\mathcal{R}%
_{\lambda}$, which is the natural domain of definition of $\lambda$.

\subsection{\label{sec2_1}The Branch Functions $\lambda_{1},\ldots,\lambda
_{n}$}

\qquad The solution $\lambda$ of the polynomial equation (\ref{f1_2e}) is a
multivalued function with $n$ branches $\lambda_{j}$, $j=1,\ldots,n$, defined
in $\overline{\mathbb{C}}$. Each pair $(\lambda,t)=(\lambda_{j}(t),t)$ with
$t\in\overline{\mathbb{C}}$, $j=1,\ldots,n$, \ satisfies the equation%
\begin{equation}
0\,=\,g(\lambda,t)\,:=\,\det\left(  \lambda\,I-(A-t\,B)\right)  \,=\,g_{(1)}%
(\lambda,t)\cdots g_{(m)}(\lambda,t),\label{f2_1a2}%
\end{equation}
which is identical with (\ref{f1_2e}), only that we now have added the
polynomials $g_{(l)}(\lambda,t)\in\mathbb{C}\left[  \lambda,t\right]  $,
$l=1,\ldots,m$, which are assumed to be irreducible. If the polynomial
$g(\lambda,t)$ itself is irreducible, then we have $m=1,$ $g(\lambda
,t)=g_{(1)}(\lambda,t)$, and $\lambda$ is an algebraic function of order $n $.
Otherwise, in case $m>1$, $\lambda$ consists of $m$ algebraic functions
$\lambda_{(l)}$, $l=1,\ldots,m$, which are defined by the $m$ polynomial
equations%
\begin{equation}
g_{(l)}(\lambda_{(l)},t)\,=\,0,\text{ \ \ \ }l=1,\ldots,m.\label{f2_1a3}%
\end{equation}
Hence, $\lambda$ consists either of a single algebraic function or of several
such functions, depending on whether $g(\lambda,t)$ is irreducible or not. In
any case, the total number of branches $\lambda_{j}$ is always exactly $n$.

Obviously, for each $t\in\mathbb{C}$, the numbers $\lambda_{1}(t),\ldots
,\lambda_{n}(t)$ are eigenvalues of the matrix $A-t\,B$, as has already been
stated in (\ref{f1_2d}). Since $A-t\,B$ is an Hermitian matrix for
$t\in\mathbb{R}$, the restriction of each branch $\lambda_{j}$, $j=1,\ldots
,n$, to $\mathbb{R}$ is a real function.

From (\ref{f2_1a2}) and the Leibniz formula for determinants we deduce that%
\begin{equation}
g(\lambda,t)\,=\,\sum_{j=0}^{n}p_{j}(t)\,\lambda^{j}\label{f2_1a5}%
\end{equation}
with $p_{j}\in\mathbb{C}\left[  t\right]  $, $\deg p_{j}\leq n-j$ for
$j=0,\ldots,n$, $p_{n}\equiv1$, and $p_{n-1}(t)=t\,\operatorname*{Tr}%
(B)-\operatorname*{Tr}(A)$. If $m>1$, then we assume the polynomials $g_{(l)}$
normalized by%
\begin{equation}
g_{(l)}(\lambda,t)\,=\,\lambda^{n_{l}}\,+\,\text{lower terms in }%
\lambda,\text{ \ \ \ }l=1,\ldots,m,\label{f2_1a4}%
\end{equation}
and we have $n_{1}+\ldots+n_{m}=n$. In situations, where we have to deal with
individual algebraic functions $\lambda_{(l)}$,\ $l=1,\ldots,m$, which will,
however, not often be the case, we denote the elements of a complete set of
branches of the algebraic function $\lambda_{(l)},$ $l=1,\ldots,m,$ by
$\lambda_{l,i}$, $i=1,\ldots,n_{l}$. There exists an obvious one-to-one
correspondence $j:\,\{\,(l,i)$, $i=1,\ldots,n_{l}$, $l=1,\ldots
,m\,\}\,\longrightarrow\,\{\,1,\ldots,n\,\}$ such that the set of functions
$\{\,\lambda_{l,i}$, $i=1,\ldots,n_{l}$, $l=1,\ldots,m\,\}$ corresponds to
$\{\,\lambda_{j},j=1,\ldots,n\,\}$ bijectively.

It belongs to the nature of branches of a multi-valued function that their
domains of definition possesses a great degree of arbitrariness. Assumptions
for limiting this freedom will be addressed in Definition \ref{sec2_def2} in
the next subsection.

Since the solution $\lambda$ of (\ref{f2_1a2}) consists either of a single or
of several algebraic functions, it is obvious that $\lambda$ possesses only
finitely many branch points over $\overline{\mathbb{C}}$.

\begin{lemma}
\label{sec2_Lem2}All branches $\lambda_{j},j=1,\ldots,n$, of the solution
$\lambda$ of (\ref{f2_1a2}) can be chosen such that they are of real type,
i.e., any function $\lambda_{j}$, which is analytic in a domain $D_{0}%
\subset\mathbb{C}$, is also analytic in the domain $D_{0}\cup
\{\,z\,|\,\overline{z}\in D_{0}\,\}$, and we have $\lambda_{j}(\overline
{t})=\overline{\lambda_{j}(t)}$ for all $t\in D_{0}$.
\end{lemma}

\begin{proof}
The relation $\lambda_{j}(\overline{t})=\overline{\lambda_{j}(t)}$\ follows
from the identity%
\[
\overline{g(\lambda,t)}=\det\left(  \overline{\lambda}\,I-(\overline
{A}-\overline{t}\,B)\right)  =\det\left(  \overline{\lambda}\,I-(\overline
{A}^{t}-\overline{t}\,B)\right)  =g(\overline{\lambda},\overline{t}),
\]
which is a consequence of $\overline{A}^{t}=A^{\ast}=A$ and of $B$ being
diagonal. Since the restriction of $\lambda_{j}$ to $\mathbb{R}$ is real,
$\overline{\lambda_{j}(\overline{t})}$ is an analytic continuation of
$\lambda_{j}$ across $\mathbb{R}$.
\end{proof}

\begin{lemma}
\label{sec2_Lem3}The solution $\lambda$ of (\ref{f2_1a2}) has no branch points
over $\mathbb{R}$.
\end{lemma}

\begin{proof}
The lemma is a consequence of the fact that the functions $\lambda
_{j},j=1,\ldots,n$, are of real type. We give an indirect proof, and assume
that $x_{0}\in\mathbb{R}$ is a branch point of order $k\geq1$ of a branch
$\lambda_{j}$, $j\in\{\,1,\ldots,n\,\}$, which we can assume to be analytic in
a slit neighborhood $V\diagdown\left(  i\mathbb{R}_{-}+x_{0}\right)  $ of
$x_{0}$. Using a local coordinate at $x_{0}$ leads to the function
$g(u):=\lambda_{j}(x_{0}+u^{k+1})$, which is analytic in a neighborhood of
$u=0$. Obviously, the function $g$ is also of real type. Let $l_{0}%
\in\mathbb{N}$ be the smallest index in the development $g(u)=\sum_{l}%
c_{l}u^{l}$ such that $c_{l_{0}}\neq0$ and $l_{0}\not \equiv 0$
$\operatorname{mod}(k+1)$, which means that there exists $0<l_{1}\leq k$ with
$l_{0}=m(k+1)+l_{1},$ $m\in\mathbb{N}$. Like $\lambda_{j}(z)=g((z-x_{0}%
)^{1/(k+1)})$, so also the modified function%
\[
\widetilde{\lambda}_{j}(z):=\left[  g((z-x_{0})^{1/(k+1)})-\sum_{l=0}%
^{m}c_{l(k+1)}(z-x_{0})^{l}\right]  \,(z-x_{0})^{-m}%
\]
has a branch point of order $k$ at $x_{0}$, and it is of real type. We have%
\[
\widetilde{\lambda}_{j}(z)=c_{l_{0}}\,(z-x_{0})^{l_{1}/(k+1)}\,+\,\text{O}%
((z-x_{0})^{(l_{1}+1)/(k+1)})\text{ \ \ as \ \ }z\rightarrow x_{0},
\]
and consequently for $r>0$ sufficiently small we have%
\[
\left\vert \arg\widetilde{\lambda}_{j}(x_{0}+r\,e^{it})-\arg c_{l_{0}}%
-\frac{l_{1}}{k+1}\,t\,\right\vert \,\leq\,\frac{\pi}{4(k+1)}\text{ \ for all
\ }0\leq t\leq\pi,
\]
which implies that%
\begin{equation}
0\,<\,\frac{l_{1}-1/2}{(k+1)}\pi\leq\left\vert \arg\widetilde{\lambda}%
_{j}(x_{0}+r)-\arg\widetilde{\lambda}_{j}(x_{0}-r)\right\vert \leq\frac
{l_{1}+1/2}{(k+1)}\pi<\pi.\text{ }\label{f2_1c}%
\end{equation}
Since the function $\widetilde{\lambda}_{j}$ is of real type, we have
$\arg\widetilde{\lambda}_{j}(x_{0}+r)\equiv0$ $\operatorname{mod}$ $\pi$ and
$\arg\widetilde{\lambda}_{j}(x_{0}-r)\equiv0$ $\operatorname{mod}$ $\pi$,
which contradicts (\ref{f2_1c}).\medskip
\end{proof}

Next, we investigate the behavior of the functions $\lambda_{j}$,
$j=1,\ldots,n$, in the neighborhood of infinity.

\begin{lemma}
\label{sec2_Lem1}Let $\lambda_{j}$, $j=1,\ldots,n$, denote $n$ different
branches of the solution $\lambda$ of (\ref{f2_1a2}). This system of branches
can be chosen in such a way that there exists a simply connected domain
$U_{\lambda}\subset\overline{\mathbb{C}}$ with $\infty\in U_{\lambda}$ such
that the following assertions hold true:

\begin{enumerate}
\item[(i)] Each function $\lambda_{j}$, $j=1,\ldots,n$, is defined throughout
$U_{\lambda}$, and none of them has a branch point in $U_{\lambda}$.

\item[(ii)] The $n$ functions $\lambda_{j}$, $j=1,\ldots,n$, can be enumerated
in such a way that at infinity we have%
\begin{equation}
\lambda_{j}(t)=a_{jj}-b_{j}t+\text{O}\left(  1/t\right)  \text{ \ \ as
\ \ }t\rightarrow\infty\label{f2_1a1}%
\end{equation}
with $a_{jj}$ and $b_{j}$, $j=1,\ldots,n$, the diagonal elements of the
matrices $A$ and $B$, respectively, of (\ref{f1b_1a}) and (\ref{f1b_1b}) in
Assumption 1.\smallskip
\end{enumerate}
\end{lemma}

\begin{remark}
Assumption 1 from Section \ref{sec1b} is decisive for the concrete form of
(\ref{f2_1a1}), and (\ref{f2_1a1}) is decisive for the verification of the
representation of the measure $\mu_{A,B}$ in Theorem \ref{sec1_thm2}, which
will follow in Subsection \ref{sec3_2} below. Notice that the similarity
transformation $(A,B)\mapsto(\widetilde{A},\widetilde{B})$ from Lemma
\ref{sec1_lem1} in general changes the diagonal elements $a_{jj},$
$j=1,\ldots,n $, of the matrix $A$, while it leaves the polynomial equation
(\ref{f2_1a2}) and also the branches $\lambda_{j}$, $j=1,\ldots,n$, invariant.
For an illustration of the changes of the $a_{jj},$ $j=1,\ldots,n$, one may
consult (\ref{f6_1a4}), where the simple case of $2\times2$ matrices has been analyzed.
\end{remark}

\begin{remark}
With Assumption 1 from Section \ref{sec1b} it is obvious that Lemma
\ref{sec1_lem2} in Section \ref{sec1_4} is a reformulation of Lemma
\ref{sec2_Lem1}.
\end{remark}

\begin{proof}
[Proof of Lemma \ref{sec2_Lem1}]We first prove that the solution $\lambda$ of
(\ref{f2_1a2}) has no branch point over infinity, which then leads to a proof
of assertion (i). The proof of assertion (ii) is more involved.\smallskip

\textit{Proof of (i)}: \ As in the proof of Lemma \ref{sec2_Lem3} we prove the
absence of a branch point at infinity indirectly, and assume that some
function $\lambda_{j}$, $j\in\{\,1,\ldots,n\,\}$, has a branch point of order
$k\geq1$ at infinity. The function $\lambda_{j}$ is of real type, and as a
branch of an algebraic function, it has at most polynomial growth for
$t\rightarrow\infty$. Hence, there exists $m_{0}\in\mathbb{N}$ such that the
function%
\[
\lambda_{0}(z)\,:=\,z^{m_{0}}\lambda_{j}(1/z)
\]
is bounded in a neighborhood of $x_{0}=0$. The function $\lambda_{0}$ is again
of real type, and it has a branch point of order $k\geq1$ at $x_{0}=0$.

After these preparations we can copy the reasoning in the proof of Lemma
\ref{sec2_Lem3} line by line in order to show that our assumption leads to a contradiction.

From equation (\ref{f2_1a2}) together with (\ref{f2_1a5}) we further deduce
that all $n$ functions $\lambda_{j}$, $j=1,\ldots,n$, are finite in
$\mathbb{C}$.

Since the solution $\lambda$ of (\ref{f2_1a2}) possesses only finitely many
branch points and none at infinity, the branches $\lambda_{1},\ldots
,\lambda_{n}$ can be chosen in such a way that there exists a punctured
neighborhood of infinity in which all $n$ functions $\lambda_{j}$,
$j=1,\ldots,n$, are defined and analytic, which concludes the proof of
assertion (i).

At infinity the functions $\lambda_{j}$, $j=1,\ldots,n$, may have a pole. In
the next part of the proof we shall see that this is indeed the case, and the
pole is always simple.\smallskip

\textit{Proof of (ii)}: \ The proof of (\ref{f2_1a1}) will be done in two
steps. In the first one we determine a condition that has to be satisfied by
the leading coefficient of the development of the function $\lambda_{j}$,
$j=1,\ldots,n$, at infinity.

Let $\lambda_{0}$ denote one of the functions $\lambda_{1},\ldots,\lambda_{n}%
$. From part (i) we know that there exists an open, simply connected
neighborhood $U_{0}\subset\overline{\mathbb{C}}$ of $\infty$ such that
$\lambda_{0}$ is analytic in $U_{0}\diagdown\{\infty\}$ and meromorphic in
$U_{0}$. Hence, $\lambda_{0}$ can be represented as%
\begin{equation}
\lambda_{0}\,=\,p\,+\,v\label{f2_1b2}%
\end{equation}
with $p$ a polynomial and $v$ a function analytic in $U_{0}$ with
$v(\infty)=0$. We will show that the polynomial $p$ is necessarily of the form%
\begin{equation}
p(t)\,=\,c_{0}-c_{1}t\text{ \ \ \ with \ \ \ }c_{1}\in\{\,b_{1},\ldots
,b_{n}\,\}.\label{f2_1b3}%
\end{equation}
The proof will be done indirectly, and we assume that%
\begin{equation}
\deg p\,\neq\,1\text{ \ \ or \ \ }p(t)=c_{0}-c_{1}t\text{\ \ with \ }%
c_{1}\notin\{\,b_{1},\ldots,b_{n}\,\}.\label{f2_1b4}%
\end{equation}
From (\ref{f2_1b4}) and the assumption made with respect to $v$ after
(\ref{f2_1b2}), it follows that%
\begin{equation}
\left\vert p(t)+b_{j}t-a_{jj}+v(t)\right\vert \rightarrow\infty\text{ \ \ as
\ }t\rightarrow\infty\text{ \ for each \ }j=1,\ldots,n.\label{f2_1b5}%
\end{equation}

From the definition of $g(\lambda,t)$ in (\ref{f2_1a2}) and the Leibniz
formula for determinants we deduce that%
\begin{align}
& g(\lambda_{0}(t),t)\,=\,\prod_{j=0}^{n}\left(  p(t)+b_{j}t-a_{jj}%
+v(t)\right)  \,+\,\label{f2_1b6}\\
& \text{ \ \ \ \ \ \ \ \ \ \ \ \ \ \ \ \ \ \ \ \ \ }+\,\text{O}\left(
\max_{j=1,\ldots,n}\left\vert p(t)+b_{j}t-a_{jj}+v(t)\right\vert
^{n-2}\right)  \text{ \ as \ }t\rightarrow\infty.\nonumber
\end{align}
Indeed, the product in (\ref{f2_1b6}) is built from the diagonal elements of
the matrix $\lambda_{0}(t)\,I-(A-t\,B)$, and any other term in the Leibniz
formula contains at least two off-diagonal elements as factors, which leads to
the error term in the second line of (\ref{f2_1b6}). From (\ref{f2_1b4}),
(\ref{f2_1b5}), and Assumption 2 in Section \ref{sec1b} we deduce that%
\[
\lim_{t\rightarrow\infty}\frac{\left\vert p(t)+b_{k}t-a_{kk}+v(t)\right\vert
}{\max_{j=1,\ldots,n}\left\vert p(t)+b_{j}t-a_{jj}+v(t)\right\vert
}\,>\,0\text{ \ \ \ for each \ \ }k=1,\ldots,n,
\]
which implies that%
\begin{equation}
\max_{j=1,\ldots,n}\left\vert p(t)+b_{j}t-a_{jj}+v(t)\right\vert ^{2-n}%
\prod_{j=0}^{n}\left\vert p(t)+b_{j}t-a_{jj}+v(t)\right\vert \rightarrow
\infty\label{f2_1b7}%
\end{equation}
as $t\rightarrow\infty$. From (\ref{f2_1b6}) together with (\ref{f2_1b5}) and
(\ref{f2_1b7}) it\ then follows that $g(\lambda_{0}(t),t)\rightarrow\infty$ as
$t\rightarrow\infty$. But this contradicts $g(\lambda_{0}(t),t)=0$ for $t\in
U_{0}$, and the contradiction proves the assertion made in (\ref{f2_1b3}%
).\medskip

We now come to the second step of the proof of (ii). Because of (\ref{f2_1b3})
we can make the ansatz%
\begin{align}
\lambda_{j}\,  & =\,p_{j}\,+\,v_{j}\text{ \ \ for \ \ }j=1,\ldots
,n,\medskip\label{f2_1c1}\\
p_{j}(t)\,  & =\,c_{0j}-c_{1j}t\text{ \ \ \ with \ \ \ }c_{1j}\in
\{\,b_{1},\ldots,b_{n}\,\},\nonumber
\end{align}
$v_{j}$ analytic in a neighborhood $U_{0}$ of infinity, and $v_{j}(\infty)=0
$. We shall show that the functions $\lambda_{1},\ldots,\lambda_{n}$ can be
enumerated in such a way that we have%
\[
c_{1j}\,=\,b_{j}\text{ \ \ and \ \ }c_{0j}=a_{jj}\text{\ \ \ for each
\ \ }j=1,\ldots,n,
\]
which proves (\ref{f2_1a1}).

A transformation of the variables $\lambda$ and $t$ into $w$ and $u$ is
introduced by%
\begin{equation}
u\,:=\,1/t\text{ \ \ and \ \ }w\,:=\,\frac{1}{\lambda+b_{1}t-a_{00}%
}\label{f2_1c2}%
\end{equation}
with%
\begin{equation}
a_{00}\,:=\,\min\left(  \{\,c_{11},\ldots,c_{1n}\,\}\cup\{\,b_{1},\ldots
,b_{n}\,\}\right)  \,-\,2.\label{f2_1c3}%
\end{equation}

From (\ref{f2_1c2}) it follows that%
\begin{equation}
\lambda\,=\,\frac{1}{w}-b_{1}t+a_{00}\,=\,\frac{1}{w}-\frac{b_{1}}{u}%
+a_{00}.\label{f2_1c4}%
\end{equation}

There exists an obvious one-to-one correspondence between the $n$ functions
$\lambda_{j}$, $j=1,\ldots,n$, and the $n$ functions%
\begin{equation}
w_{j}(u)\,:=\,\frac{1}{\lambda_{j}(1/u)+b_{1}/u-a_{00}},\text{ \ \ }%
j=1,\ldots,n\text{.}\label{f2_1c5}%
\end{equation}
The functions $w_{j}$, $j=1,\ldots,n$, are meromorphic in a neighborhood
$\widetilde{U}_{0}$ of the origin. From (\ref{f2_1c1}) and (\ref{f2_1c5}) we
deduce that%
\begin{equation}
w_{j}(0)\,=\,\left\{
\begin{array}
[c]{lll}%
0\medskip & \text{ \ for \ } & c_{1j}\neq b_{1}\\
\displaystyle\frac{1}{c_{0j}-a_{00}}\,\leq\,\frac{1}{2} & \text{ \ for \ } &
c_{1j}=b_{1},
\end{array}
\right. \label{f2_1c6}%
\end{equation}
and therefore we can choose $\widetilde{U}_{0}$ so small that%
\begin{equation}
0\,<\,|w_{j}(u)|\,\leq\,1\text{ \ \ for \ \ }u\in\widetilde{U}_{0}%
\diagdown\{0\},\label{f2_1c7}%
\end{equation}
which implies that all $w_{j}$, $j=1,\ldots,n$, are analytic in $\widetilde
{U}_{0}$.$\smallskip$

By $V(u)$, $u\in\mathbb{C}\diagdown\{0\}$, we denote the $n\times n$ diagonal
matrix%
\begin{equation}
V(u)\,:=\,\operatorname*{diag}(\,\underset{m_{1}}{\underbrace{1,\ldots,1}%
},\underset{n-m_{1}}{\underbrace{\sqrt{u},\ldots,\sqrt{u}}}\,),\label{f2_1c8}%
\end{equation}
where $m_{1}$ is the number of appearances of $b_{1}$ in the multiset
$\{\,b_{1},\ldots,b_{n}\,\}=\{\,b_{j},j=1,\ldots,n\,\}$, and we define%
\begin{equation}
\widetilde{g}(w,u)\,:=\,\det\left(  V(u)^{2}+w\,(B-b_{1}I)-w\,V(u)(A-a_{00}%
I)V(u)\right)  .\label{f2_1c9}%
\end{equation}
We then deduce that%
\begin{align}
\widetilde{g}(w,u)  & =\det\left(  V(u)\left(  I+\frac{w}{u}\,(B-b_{1}%
I)-w\,(A-a_{00}I)\right)  V(u)\right)  \medskip\nonumber\\
& =w^{n}u^{n-m_{1}}\det\left(  \frac{1}{w}I+\frac{1}{u}\,(B-b_{1}%
I)-(A-a_{00}I)\right)  \medskip\label{f2_1c10}\\
& =w^{n}u^{n-m_{1}}\det\left(  \left(  \frac{1}{w}-\frac{b_{1}}{u}%
+a_{00}\right)  \,I-\left(  A-\frac{1}{u}B\right)  \right)  \medskip
\nonumber\\
& =w^{n}u^{n-m_{1}}g(\lambda,\frac{1}{u}).\nonumber
\end{align}
Indeed, the first equality is obvious if we take into account that
$B-b_{1}I\,=\,\operatorname*{diag}(\,0,\ldots,0,b_{m_{1}+1}-b_{1},\ldots
,b_{n}-b_{1}\,)$ with exactly $m_{1}$ zeros in its diagonal. The next three
equations result from elementary transformations.

Directly from (\ref{f2_1c9}), but also from (\ref{f2_1a5}) and (\ref{f2_1c10})
together with (\ref{f2_1c4}) we deduce that $\widetilde{g}(w,u)$ is a
polynomial in $w$ and $u$, and is of order $n$ in $w$.

From (\ref{f2_1c9}) together with properties used in (\ref{f2_1c10}) and the
Leibniz formula for determinants it follows that%
\begin{align}
& \widetilde{g}(w,u)\,=\,\prod_{j=1}^{m_{1}}\left(  1-w(a_{jj}-a_{00})\right)
\prod_{j=m_{1}+1}^{n}\left(  u-w(b_{j}-b_{1})-w\,u(a_{jj}-a_{00})\right)
\times\medskip\nonumber\\
& \text{
\ \ \ \ \ \ \ \ \ \ \ \ \ \ \ \ \ \ \ \ \ \ \ \ \ \ \ \ \ \ \ \ \ \ \ \ \ \ \ \ \ \ \ \ \ \ \ \ \ \ \ \ \ \ \ }%
\times(1\,+\,\text{O}\left(  u\right)  )\text{ \ as \ }u\rightarrow
0.\label{f2_1c11}%
\end{align}
Indeed, the product in (\ref{f2_1c11}) is formed by the diagonal elements of
the matrix $M:=V(u)^{2}+w\,(B-b_{1}I)-w\,V(u)(A-a_{00}I)V(u)$, and the error
term O$\left(  u\right)  $ in the second line of (\ref{f2_1c11}) results from
the fact that each other term in the Leibniz formula includes at least two
off-diagonal elements of the matrix $M$ as factors. Each off-diagonal element
of $M$ contains the factor $\sqrt{u}$, or it is zero since from Assumption 1
in Section \ref{sec1b} it follows that for all elements $m_{ij}$ of
$M=(m_{ij})$ with $i,j=1,\ldots,m_{1},$ $i\neq j$, we have $m_{ij}%
=0$.$\medskip$

With (\ref{f2_1c11}) we are prepared to describe the behavior of the functions
$w_{1},\ldots,w_{n}$ near $u=0$, which then translates into a proof of the
first part of (\ref{f2_1a1}).

For each $u\in\mathbb{C}$ the $n$ values $w_{1}(u),\ldots,w_{n}(u)$ are the
zeros of the polynomial $\widetilde{g}(w,u)\in\mathbb{C}\left[  w\right]  $.
From (\ref{f2_1c11}) we know that%
\[
\widetilde{g}(w,u)\,\rightarrow\,w^{n-m_{1}}\prod_{j=1}^{m_{1}}\left(
1-w(a_{jj}-a_{00})\right)  \prod_{j=m_{1}+1}^{n}(b_{j}-b_{1})\text{ \ \ as
\ \ }u\rightarrow0.
\]
Therefore it follows by Rouch\'{e}'s Theorem that with an appropriate
enumeration of the functions $w_{j}$, $j=1,\ldots,n$, we have%
\begin{equation}
\lim_{u\rightarrow0}w_{j}(u)\,=\,\left\{
\begin{array}
[c]{lll}%
\displaystyle\frac{1}{a_{jj}-a_{00}}\medskip & \text{ \ for \ } &
j=1,\ldots,m_{1}\\
0 & \text{ \ for \ } & j=m_{1}+1,\ldots,n,
\end{array}
\right. \label{f2_1d1}%
\end{equation}
which is a concretization of (\ref{f2_1c6}). Since we know from (\ref{f2_1c7})
that all functions $w_{j}$, $j=1,\ldots,n$, are analytic in a neighborhood
$\widetilde{U}_{0}$ of the origin, it follows from (\ref{f2_1d1}) that%
\begin{equation}
w_{j}(u)\,=\,\frac{1}{a_{jj}-a_{00}}+\text{O}(u)\text{ \ \ as \ \ }%
u\rightarrow0\text{ \ for \ }j=1,\ldots,m_{1}.\label{f2_1d2}%
\end{equation}

From the correspondence (\ref{f2_1c5}) between the functions $w_{j}$,
$j=1,\ldots,n$, and $\lambda_{j}$, $j=1,\ldots,n$, it then follows from
(\ref{f2_1d2}) that%
\begin{align}
\lambda_{j}(t)  & =\frac{1}{w_{j}(1/t)}-b_{1}t+a_{00}\medskip\nonumber\\
& =a_{jj}-a_{00}-b_{1}t+a_{00}+\text{O}(\frac{1}{t})\medskip\label{f2_1d3}\\
& =a_{jj}-b_{j}t+\text{O}(\frac{1}{t})\text{ \ \ as \ \ }t\rightarrow
\infty\text{ \ for \ }j=1,\ldots,m_{1}.\nonumber
\end{align}
The last equation is a consequence of $b_{j}=b_{1}$ for $j=1,\ldots,m_{1}$.
With (\ref{f2_1d3}) we have proved relation (\ref{f2_1a1}) for $j=1,\ldots
,m_{1}$.

By the definition of $m_{1}$ and the ordering in (\ref{f1b_2a}) we have%
\[
b_{m_{1}+1}\,>\,b_{m_{1}}=\cdots=b_{1}.
\]
Let now $m_{2}$ denote the number of appearances of the value$\ b_{m_{1}+1}$
in the multiset $\{\,b_{j},j=1,\ldots,n\,\}$. In order to prove relation
(\ref{f2_1d3}) for $j=m_{1}+1,\ldots,m_{1}+m_{2}$, we repeat the analysis from
(\ref{f2_1c2}) until (\ref{f2_1d3}) with, $b_{1}$ replaced by $b_{m_{1}+1}$
and $m_{1}$ by $m_{2}$, which then leads to the verification of (\ref{f2_1d3})
for $j=m_{1}+1,\ldots,m_{1}+m_{2}$.

Repeating this cycle for each different value$\,b_{j}$ in the multiset
$\{\,b_{j},$ $\allowbreak j=1,\ldots,n\,\}$ proves relation (\ref{f2_1d3}) for
all $j=1,\ldots,n$, which completes the proof of (\ref{f2_1a1}), and concludes
the proof of assertion (ii).

We would like to add as a short remark that if all $b_{j}$, $j=1,\ldots,n$,
are pairwise different, then the analysis in these last cycles could be
considerably shortened since in such a case one could proceed rather directly
from (\ref{f2_1c6}) to the conclusion (\ref{f2_1d3}).\smallskip
\end{proof}

\subsection{\label{sec2_2}The Complex Manifold $\mathcal{R}_{\lambda}$}

\qquad If the polynomial $g(\lambda,t)$ in (\ref{f2_1a2}) is irreducible, then
the solution $\lambda$ of (\ref{f2_1a2}) is an algebraic function of order
$n$, and its natural domain of definition is a compact Riemann surface with
$n$ sheets over $\overline{\mathbb{C}}$ (cf. \cite[Theorem IV.11.4]%
{FarkasKra92}). We denote this surface by $\mathcal{R}_{\lambda}$.

If, however, the polynomial $g(\lambda,t)$ is reducible, then we have seen in
(\ref{f2_1a2}) and (\ref{f2_1a3}) that the solution $\lambda$ of
(\ref{f2_1a2}) consists of $m$ algebraic functions $\lambda_{(l)}$,
$l=1,\ldots,m$. Each $\lambda_{(l)}$ has a compact Riemann surface
$\mathcal{R}_{\lambda,l}$, $l=1,\ldots,m$, as its natural domain of
definition, and therefore the complex manifold%
\begin{equation}
\mathcal{R}_{\lambda}\,:=\,\mathcal{R}_{\lambda,1}\,\cup\cdots\cup
\,\mathcal{R}_{\lambda,m}\label{f2_2a1}%
\end{equation}
is the natural domain of definition for the multivalued function $\lambda$. In
each of the two cases, $\mathcal{R}_{\lambda}$ is a covering of $\overline
{\mathbb{C}}$ with exactly $n$ sheets, except that in the later case
$\mathcal{R}_{\lambda}$ is no longer connected. By $\pi_{\lambda}%
:\mathcal{R}_{\lambda}\longrightarrow\overline{\mathbb{C}}$ we denote the
canonical projection of $\mathcal{R}_{\lambda}$.

A collection of subsets $\left\{  S_{\lambda}^{(j)}\subset\mathcal{R}%
_{\lambda},\text{ }j=1,\ldots,n\right\}  $ is called a system of sheets on
$\mathcal{R}_{\lambda}$ if the following three requirements are
satisfied:\medskip

\noindent(i) \ \ The restriction $\left.  \pi_{\lambda}\right\vert
_{S_{\lambda}^{(j)}}:S_{\lambda}^{(j)}\longrightarrow\overline{\mathbb{C}} $
of the canonical projection $\pi_{\lambda}$ is a bijection for each
$j=1,\ldots,n$.\medskip

\noindent(ii) \ We have $\bigcup\nolimits_{j=1}^{n}S_{\lambda}^{(j)}%
=\mathcal{R}_{\lambda}$.\medskip

\noindent(iii) The interior points of each sheet $S_{\lambda}^{(j)}%
\subset\mathcal{R}_{\lambda}$, $j=1,\ldots,n$, form a domain. Different sheets
are disjoint except for branch points. A branch point of order $k\geq1$
belongs to exactly $k+1$ sheets.\medskip

Because of requirement (i) each sheet $S_{\lambda}^{(j)}$ can be identified
with $\overline{\mathbb{C}}$, however, formally we consider it as a subset of
$\mathcal{R}_{\lambda}$.

While the association of branch points and sheets is specified completely in
requirement (iii), there remains freedom with respect to the other boundary
points of the sheets. \ We assume that this association is done in a pragmatic
way. It is only required that each boundary point belongs to exactly one sheet
if it is not a branch point.

Requirement (i) justifies the notational convention that a point of
$S_{\lambda}^{(j)}$ is denoted by $t^{(j)}$ if $\pi_{\lambda}(t^{(j)}%
)=t\in\overline{\mathbb{C}}$.\smallskip

The requirements (i) - (iii) give considerable freedom for choosing a system
of sheets on $\mathcal{R}_{\lambda}$. In order to get unambiguity up to
boundary associations, we define a standard system of sheets by the following
additional requirement.\medskip

\noindent(iv) \ \ The cuts, which separate different sheets $S_{\lambda}%
^{(j)}$ in $\mathcal{R}_{\lambda}$, lie over lines in $\mathbb{C}$ that are
perpendicular to $\mathbb{R}$. Each cut is chosen in a minimal way. Hence, it
begins and ends with a branch point.\smallskip

\begin{lemma}
\label{sec2_Lem4}There exists a system of sheets $S_{\lambda}^{(j)}%
\subset\mathcal{R}_{\lambda}$, $j=1,\ldots,n$, that satisfies the requirements
(i) through (iv). Such a system is essentially unique, i.e., unique up to the
association of boundary points that are not branch points. The domain
$U_{\lambda}$ from Lemma \ref{sec2_Lem1} can be chosen in such a way that each
sheet $S_{\lambda}^{(j)}$, $j=1,\ldots,n$, of the standard system covers
$U_{\lambda}$, i.e., we have%
\begin{equation}
\pi_{\lambda}(\operatorname*{Int}(S_{\lambda}^{(j)}))\supset U_{\lambda
}.\label{f2_2a2}%
\end{equation}

\end{lemma}

\begin{proof}
From part (i) of Lemma \ref{sec2_Lem1} it is evident that there exist $n$
unramified subdomains in $\mathcal{R}_{\lambda}$ over the domain $U_{\lambda}%
$; they are given by the set $\pi_{\lambda}^{-1}(U_{\lambda})$. We can choose
$U_{\lambda}\subset\overline{\mathbb{C}}$ as a disc around $\infty$. Because
of Lemmas \ref{sec2_Lem2} and \ref{sec2_Lem3} it is then always possible to
start an analytic continuation of a given branch $\lambda_{j}$, $j=1,\ldots
,n$, at $\infty$ and continue along rays that are perpendicular to
$\mathbb{R}$ until one hits a branch point or the real axis. The earlier case
can happen only finitely many times. Each of these continuations then defines
a sheet $S_{\lambda}^{(j)}$, and the whole system satisfies the requirements
(i) through (iv), and also (\ref{f2_2a2}) is satisfied.\medskip
\end{proof}

Each system $\left\{  S_{\lambda}^{(j)}\subset\mathcal{R}_{\lambda},\text{
}j=1,\ldots,n\right\}  $ of sheets corresponds to a complete system of
branches $\lambda_{j}$, $j=1,\ldots,n$, of the solution $\lambda$ of
(\ref{f2_1a2}) if we define the functions $\lambda_{j}$ by%
\begin{equation}
\lambda_{j}:=\lambda\circ\pi_{t,j}^{-1},\text{ \ \ }j=1,\ldots
,n,\label{f2_2b1}%
\end{equation}
with $\pi_{\lambda,j}^{-1}$ denoting the inverse of $\left.  \pi_{\lambda
}\right\vert _{S_{\lambda}^{(j)}}$, which exists because of requirement (i).
If we use the standard system of sheets, then the branches $\lambda_{j} $,
$j=1,\ldots,n$, are uniquely defined functions.\smallskip

\begin{definition}
\label{sec2_def2}In the sequel we denote by $\lambda_{j}$, $j=1,\ldots,n$, the
$n$ branches of the solution $\lambda$ of equation (\ref{f2_1a2}) that are
defined by (\ref{f2_2b1}) with the standard system $\left\{  S_{\lambda}%
^{(j)}\right\}  $ of sheets.\smallskip
\end{definition}

The next Lemma is an immediate consequence of the Monodromy Theorem.

\begin{lemma}
\label{sec2_Lem5}Let $\lambda_{j}$, $j=1,\ldots,n$, be the functions from
Definition \ref{sec2_def2}. Then for any entire function $g$ the function%
\[
G(t)\,=\,\sum_{j=1}^{n}g(\lambda_{j}(t))\text{, \ \ \ \ }t\in\mathbb{C}%
\text{,}%
\]
is analytic and single-valued throughout $\mathbb{C}$.
\end{lemma}

With the functions $\lambda_{j}$, $j=1,\ldots,n$, we get a very helpful
representation of the function $f$ from (\ref{f1_1a0}) and also of the
determinant $\det\left(  \zeta I-(A-tB)\right)  $.

\begin{lemma}
\label{sec2_Lem5a}With the functions $\lambda_{j}$, $j=1,\ldots,n$, from
Definition \ref{sec2_def2}, the function $f$ from (\ref{f1_1a0}) can be
represented as%
\begin{equation}
f(t)\,=\,\operatorname*{Tr}e^{A-tB}\,=\,\sum_{j=1}^{n}e^{\lambda_{j}(t)}\text{
\ \ for \ \ }t\in\mathbb{C}.\label{f2_2b2}%
\end{equation}
It follows from Lemma \ref{sec2_Lem5} that $f$ is an entire function.
\end{lemma}

\begin{proof}
From equation (\ref{f2_1a2}) it follows that for any $t\in\mathbb{C}$ the $n $
numbers $\lambda_{1}(t),\ldots,\lambda_{n}(t)$ are the eigenvalues of the the
matrix $A-t\,B$. Let $V_{\lambda}\subset\mathbb{C}$ be the set of all
$t\in\mathbb{C}$ such that not all $\lambda_{1}(t),\ldots,\lambda_{n}(t)$ are
pairwise different. This set is finite. For every $t\in\mathbb{C}\diagdown
V_{\lambda}$ the $n$ eigenvectors corresponding to $\lambda_{1}(t),\ldots
,\lambda_{n}(t)$ form an eigenbasis. The $n\times n$ matrix $T_{0}=T_{0}(t)$
with these vectors as columns satisfies%
\begin{equation}
T_{0}^{-1}(A-t\,B)T_{0}=\operatorname*{diag}\left(  \lambda_{1}(t),\ldots
,\lambda_{n}(t)\right)  .\label{f2_2b3}%
\end{equation}
Since the trace of a square matrix is invariant under similarity
transformations, (\ref{f2_2b2}) follows from (\ref{f2_2b3}) and (\ref{f1_1a0})
for $t\notin V_{\lambda}$, and by continuity for all $t\in\mathbb{C}$.\medskip
\end{proof}

\begin{lemma}
\label{sec2_Lem5b}With the functions $\lambda_{j}$, $j=1,\ldots,n$, from
Definition \ref{sec2_def2} we have%
\begin{equation}
\prod_{j=1}^{n}(\zeta-\lambda_{j}(t))\,=\,\det\left(  \zeta I-(A-tB)\right)
\text{ \ \ for \ \ }\zeta,t\in\mathbb{C}.\label{f2_2b4}%
\end{equation}

\end{lemma}

\begin{proof}
From (\ref{f2_2b3}) we deduce that%
\[
T_{0}^{-1}\left(  \zeta I-(A-tB)\right)  T_{0}=\operatorname*{diag}\left(
\zeta-\lambda_{1}(t),\ldots,\zeta-\lambda_{n}(t)\right)
\]
for each $\zeta\in\mathbb{C}$ and $t\in\mathbb{C}\diagdown V_{\lambda}$, which
then proves (\ref{f2_2b4}).\medskip
\end{proof}

In the last lemma of the present section we lift the complex conjugation from
$\overline{\mathbb{C}}$ to $\mathcal{R}_{\lambda}$.

\begin{lemma}
\label{sec2_Lem5c}There exists a unique anti-holomorphic mapping
$\rho:\mathcal{R}_{\lambda}\longrightarrow\mathcal{R}_{\lambda}$ such that we
have%
\begin{equation}
\pi_{\lambda}\circ\rho(z)\,=\,\overline{\pi_{\lambda}(z)}\text{ \ \ for all
\ \ }z\in\mathcal{R}_{\lambda}\label{f2_2c1}%
\end{equation}
and that $\rho|_{\pi_{\lambda}^{-1}(\mathbb{R})}$ is the identity.
\end{lemma}

\begin{proof}
We start with the problem of existence. Because of requirement (i) of the
standard system of sheets $\{S_{\lambda}^{(j)}\}$ on $\mathcal{R}_{\lambda} $,
we can define $\rho$ on each $S_{\lambda}^{(j)},j=1,\ldots,n$, by a direct
transfer of the complex conjugation from $\overline{\mathbb{C}}$ to
$S_{\lambda}^{(j)}$. Notice that each $\pi_{\lambda}(S_{\lambda}%
^{(j)}),j=1,\ldots,n$, is invariant under complex conjugation because of
requirement (iv) and since each $\lambda_{j}$ is of real type. It is not
difficult to see that this piecewise definition of $\rho$ is well defined
throughout $\mathcal{R}_{\lambda}$, and possesses the required properties.

The uniqueness of $\rho$ is a consequence of the fact that $\rho
|_{\pi_{\lambda}^{-1}(\mathbb{R})}$ is the identity map. Indeed, let $\rho
_{1}$ and $\rho_{2}$ be two maps with the required properties. Then $\rho
_{1}\circ\rho_{1}$ and $\rho_{1}\circ\rho_{2}$ are both analytic maps from
$\mathcal{R}_{\lambda}$ to $\mathcal{R}_{\lambda}$. On $\pi_{\lambda}%
^{-1}(\mathbb{R})$ both maps are the identity, and consequently $\rho_{1}%
\circ\rho_{1}$ and $\rho_{1}\circ\rho_{2}$ are both the identity map on
$\mathcal{R}_{\lambda}$, which proves $\rho_{1}=\rho_{2}$.
\end{proof}

\section{\label{sec3}First Part of the Proof of Theorem \ref{sec1_thm2}}

\qquad In the present section we prove all assertions of Theorem
\ref{sec1_thm2} except for the positivity of the measure $\mu_{A,B}$, which
will be the topic of the next section.

\subsection{\label{sec3_1}Equivalence of (\ref{f1_3b}) and (\ref{f1_3c})}

\begin{lemma}
\label{sec3_lem1}For each $t>0$ we have%
\begin{equation}
\sum_{j=1}^{n}\frac{1}{2\pi i}\oint\nolimits_{C_{j}}e^{\lambda_{j}\left(
\zeta\right)  +t\,\zeta}d\zeta\,=\,0\label{f3_1a}%
\end{equation}
with $C_{j}$ and $\lambda_{j}$ as specified in Theorem \ref{sec1_thm2}.
\end{lemma}

\begin{proof}
From Lemma \ref{sec2_Lem1} it is obvious that we can choose all $C_{j}$,
$j=1,\ldots,n$, to be identical with a single curve $C\subseteq\mathbb{C}$
such that all $\lambda_{1},\ldots,\lambda_{n}$ are analytic on and outside of
$C$. We interchange summation and integration in (\ref{f3_1a}), and deduce
from Lemma \ref{sec2_Lem5} that $\sum_{j=1}^{n}e^{\lambda_{j}\left(
\zeta\right)  +t\,\zeta}\,=\,e^{t\,\zeta}\sum_{j=1}^{n}e^{\lambda_{j}\left(
\zeta\right)  }$ is an entire function, which proves (\ref{f3_1a}).
\end{proof}

From (\ref{f3_1a}) it follows immediately that the representations
(\ref{f1_3b}) and (\ref{f1_3c}) in Theorem \ref{sec1_thm2} for the density
function $w_{A,B}$ are equivalent.

\subsection{\label{sec3_2}Proof of (\ref{f1_3a}), (\ref{f1_3b}), and
(\ref{f1_3c})}

\qquad We use (\ref{f1_3a}) and (\ref{f1_3b}) in Theorem \ref{sec1_thm2} as an
ansatz for a measure $\mu_{A,B}$ and show by direct calculations that this
measure satisfies (\ref{f1_1a}).

From (\ref{f1_3b}) it is evident that $w_{A,B}(t)=0$ for $0\leq t<b_{1}$; and
since we know from the last subsection that (\ref{f1_3b}) and (\ref{f1_3c})
are equivalent representations, we further deduce from (\ref{f1_3c}) that also
$w_{A,B}(t)=0$ for $t>b_{n}$. From (\ref{f1_3a}) and (\ref{f1_3b}) we then get%
\begin{equation}
\int e^{-t\,s}d\mu_{A,B}(s)\,=\,\sum_{j=1}^{n}e^{a_{jj}}e^{-t\,b_{j}}%
+\sum_{k=1}^{n-1}I_{k}(t)\text{ \ \ with}\label{f3_2a}%
\end{equation}%
\begin{equation}
I_{k}(t)\,=\,\int_{b_{k}}^{b_{k+1}}\sum_{j=1}^{k}\frac{1}{2\pi i}%
\oint\nolimits_{C_{j}}e^{\lambda_{j}(\zeta)+s(\zeta-t)}d\zeta ds,\text{
\ \ }k=1,\ldots,n-1.\label{f3_2b}%
\end{equation}
As in the proof of Lemma \ref{sec3_lem1} we assume again that all integration
paths $C_{j}$, $j=1,\ldots,n$, in (\ref{f3_2b})\ are identical with a single
curve $C\subseteq\mathbb{C}$ such that all $\lambda_{1},\ldots,\lambda_{n}$
are analytic on and outside of $C$ with a simple pole at infinity. Because of
Lemma \ref{sec2_Lem3} we can assume that%
\begin{equation}
\mathbb{R}_{+}\,\subset\,\operatorname*{Ext}(C).\label{f3_2c}%
\end{equation}
After these preparations we deduce from (\ref{f3_2b}) that%
\begin{align}
\sum_{k=1}^{n-1}I_{k}(t)\,  & =\,\sum_{k=1}^{n-1}\frac{1}{2\pi i}%
\oint\nolimits_{C}e^{\lambda_{k}(\zeta)}\int_{b_{k}}^{b_{n}}e^{s(\zeta
-t)}dsd\zeta\text{ \ \ \ \ \ \ \ \ \ \ \ \ }\nonumber\\
\text{\ \ \ \ \ \ \ \ \ \ \ \ }  & =\,\sum_{k=1}^{n-1}\frac{1}{2\pi i}%
\oint\nolimits_{C}e^{\lambda_{k}(\zeta)}\left[  e^{b_{n}(\zeta-t)}%
-e^{b_{k}(\zeta-t)}\right]  \frac{d\zeta}{\zeta-t}\text{ \ \ \ \ \ \ \ }%
\nonumber\\
\text{\ \ \ \ \ \ \ \ \ \ \ \ }  & =\,\sum_{k=1}^{n}\frac{-1}{2\pi i}%
\oint\nolimits_{C}e^{\lambda_{k}(\zeta)}e^{b_{k}(\zeta-t)}\frac{d\zeta}%
{\zeta-t}\label{f3_2d}\\
\text{\ \ \ \ \ \ \ \ \ \ \ \ }  & =\,\sum_{k=1}^{n}\left(  e^{\lambda_{k}%
(t)}\,-\,e^{a_{kk}-t\,b_{k})}\right)  .\nonumber
\end{align}
Indeed, the first equality in (\ref{f3_2d}) is a consequence of Fubini's
Theorem and (\ref{f3_2b}), the second one follows from elementary integration,
and the third one follows in the same way as the conclusion in the proof of
Lemma \ref{sec3_lem1}. We give some more details, and deduce with the help of
Lemma \ref{sec2_Lem5} that%
\[
\sum_{k=1}^{n}\frac{1}{2\pi i}\oint\nolimits_{C}e^{\lambda_{k}(\zeta)}%
e^{b_{n}(\zeta-t)}\frac{d\zeta}{\zeta-t}\,=\,\frac{1}{2\pi i}\oint
\nolimits_{C}e^{b_{n}(\zeta-t)}\sum_{k=1}^{n}e^{\lambda_{k}(\zeta)}%
\frac{d\zeta}{\zeta-t}\,=\,0,
\]
which then proves the third equality in (\ref{f3_2d}). Notice that
$t\in\operatorname*{Ext}(C)$. For a verification of the last equality in
(\ref{f3_2d}) we define the functions $r_{k}$, $k=1,\ldots,n$, by%
\[
\lambda_{k}(z)+b_{k}z\,=\,a_{kk}+r_{k}(z).
\]
It then follows from (\ref{f2_1a1}) in Lemma \ref{sec2_Lem1} that
$r_{k}(\infty)=0$ for $k=1,\ldots,n$, and obviously each $r_{k}$ is analytic
on and outside of $C$. Since $C$ is positively oriented, it follows from
Cauchy's formula that%
\begin{align*}
& \frac{-1}{2\pi i}\oint\nolimits_{C}e^{\lambda_{k}(\zeta)}e^{b_{k}(\zeta
-t)}\frac{d\zeta}{\zeta-t}\,=\,\frac{-e^{-t\,b_{k}}}{2\pi i}\oint
\nolimits_{C}e^{\lambda_{k}(\zeta)+b_{k}\zeta}\frac{d\zeta}{\zeta-t}\medskip\\
& \text{ \ \ \ \ \ }=\,\frac{-e^{a_{kk}-t\,b_{k}}}{2\pi i}\oint\nolimits_{C}%
e^{r_{k}(\zeta)}\frac{d\zeta}{\zeta-t}\medskip\\
& \text{ \ \ \ \ \ }=\text{\thinspace}e^{a_{kk}-t\,b_{k}}\left(  e^{r_{k}%
(t)}\,-\,1\right)  \,=\,e^{\lambda_{k}(t)}\,-\,e^{a_{kk}-t\,b_{k}}%
\end{align*}
for each $k=1,\ldots,n$, which completes the verification of the last equality
in (\ref{f3_2d}).

By putting (\ref{f3_2a}) and (\ref{f3_2d}) together we arrive at
(\ref{f1_1a}), which proves that (\ref{f1_3a}) and (\ref{f1_3b}) is a
representation of the measure $\mu_{A,B}$ that satisfies (\ref{f1_1a}). From
Subsection \ref{sec3_1} it then follows that also (\ref{f1_3a}) in combination
with (\ref{f1_3c}) defines the same measure $\mu_{A,B}$.

\subsection{\label{sec3_3}Proof of the Inclusion (\ref{f1_3d})}

\qquad Since before (\ref{f3_2a}) we have verified that $w_{A,B}(t)=0$ for
$0\leq t<b_{1}$ and for $t>b_{n}$, inclusion (\ref{f1_3d}) in Theorem
\ref{sec1_thm2} follows from (\ref{f1_3a}).

From (\ref{f3_2b}) it is immediately obvious that the density function
$w_{A,B}$ is the restriction of an entire function in each interval of the set
$[b_{1},b_{n}]\diagdown\{b_{1},\ldots,b_{n}\}$.

\subsection{Remark about the Proof of (\ref{f1_3a}), (\ref{f1_3b}), and
(\ref{f1_3c})}

\qquad In Subsection \ref{sec3_2} the representation of the measure $\mu
_{A,B}$ in Theorem \ref{sec1_thm2} has been proved with the help of an ansatz.
This strategy is very effective, but it gives no hints how one can
systematically find such an ansatz. Actually, the expressions in (\ref{f1_3a})
and (\ref{f1_3b}) were only found after a lengthy asymptotic analysis of the
function (\ref{f1_1a0}) with a subsequent application of the Post-Widder
formulae for the inversion of Laplace transforms. This systematic, but
laborious approach is posted at the ArXiv under \cite[Version 2]{Stahl11}.

\section{\label{sec5}The Proof of Positivity}

For the completion of the proof of Theorem \ref{sec1_thm2} it remains only to
show that the measure $\mu_{A,B}$ is positive, which is done in the present
section. The essential problem is to show that the density function $w_{A,B}$
given by (\ref{f1_3b}) or by (\ref{f1_3c}) in Theorem \ref{sec1_thm2} is
non-negative in $\left[  b_{1},b_{n}\right]  \diagdown\{b_{1},\ldots,b_{n}\}$.

\subsection{\label{sec5_1}A Preliminary Assumption}

\qquad In a first version of the proof of positivity we make the following
additional assumption, which will afterwards, in Subsection \ref{sec5_4}, be
shown to be superfluous.$\medskip$

\noindent\textbf{Assumption 3}. \textit{We assume that the polynomial
}$g(\lambda,t)$\textit{\ in equation (\ref{f2_1a2}), which is identical with
the polynomial in (\ref{f1_2e}), is irreducible.}$\medskip$

For the convenience of the reader we list definitions from Section \ref{sec2}
that will be especially important in the next subsection. Some of them now
have special properties because of Assumption 3.

\begin{enumerate}
\item[(i)] The solution $\lambda$ of equation\textit{\ }(\ref{f2_1a2}) is an
algebraic function of degree $n$ (cf. Subsection \ref{sec2_1}).

\item[(ii)] The covering manifold $\mathcal{R}_{\lambda}$ over $\overline
{\mathbb{C}}$\ from Subsection \ref{sec2_2} is now a compact Riemann surface
with $n$ sheets over $\overline{\mathbb{C}}$. As before, by $\pi_{\lambda
}:\mathcal{R}_{\lambda}\longrightarrow\overline{\mathbb{C}}$ we denote its
canonical projection.

\item[(iii)] The $n$ functions $\lambda_{j}$, $j=1,\ldots,n$, from Definition
\ref{sec2_def2} in Subsection \ref{sec2_2} are $n$ branches of the single
algebraic function $\lambda$.

\item[(iv)] By $C_{j}$, $j=1,\ldots,n$, we denote $n$ Jordan curves that are
all identical with a single curve $C\subset\mathbb{C}$, and this curve is
assumed to be smooth, positively oriented, and chosen in such a way that each
function $\lambda_{j}$, $j=1,\ldots,n$, is analytic on and outside of $C$.

\item[(v)] The reflection function $\varrho:\mathcal{R}_{\lambda
}\longrightarrow\mathcal{R}_{\lambda}$ from Lemma \ref{sec2_Lem5c} in
Subsection \ref{sec2_2} is the lifting of the complex conjugation from
$\overline{\mathbb{C}}$ onto $\mathcal{R}_{\lambda}$, i.e., we have
$\pi_{\lambda}(\varrho(\zeta))=\overline{\pi_{\lambda}(\zeta)}$ for all
$\zeta\in\mathcal{R}_{\lambda}$. By $\mathcal{R}_{+}\subset\mathcal{R}%
_{\lambda}$ we denote the subsurface $\mathcal{R}_{+}:=\{\,z\in\mathcal{R}%
_{\lambda}\,|\,\operatorname*{Im}\pi_{\lambda}(z)>0\,\}$, and by
$\mathcal{R}_{-}\subset\mathcal{R}_{\lambda}$ the corresponding subsurface
defined over base points with a negative imaginary part; $\mathcal{R}_{+}$ and
$\mathcal{R}_{-}$ are bordered Riemann surfaces over $\{\,\operatorname*{Im}%
\,z>0\,\}$ and $\{\,\operatorname*{Im}\,z<0\,\}$, respectively.
\end{enumerate}

\subsection{\label{sec5_2}The Main Proposition}

\qquad The proof of positivity under Assumption 3 is based on assertions that
are formulated and proved in the next proposition.

\begin{proposition}
\label{sec5_prop1}Under Assumption 3 for any $t\in(b_{I},b_{I+1})$ with
$I\in\{1,\ldots,n\allowbreak-1\}$ there exists a chain $\gamma$ of finitely
many closed integration paths on the Riemann surface $\mathcal{R}_{\lambda}$
such that%
\begin{equation}
\operatorname*{Im}e^{\lambda(\zeta)+t\,\pi_{\lambda}(\zeta)}\,=\,0\text{
\ \ \ for all \ \ \ }\zeta\in\gamma,\medskip\label{f5_2a1}%
\end{equation}%
\begin{equation}
\frac{1}{2\pi i}\oint\nolimits_{\gamma}e^{\lambda(\zeta)+t\,\pi_{\lambda
}(\zeta)}d\zeta\,<\,0,\smallskip\label{f5_2a2}%
\end{equation}%
\begin{equation}
\frac{1}{2\pi i}\oint\nolimits_{\gamma}e^{\lambda(\zeta)+t\,\pi_{\lambda
}(\zeta)}d\zeta\,=\,-\,\sum_{j=1}^{I}\frac{1}{2\pi i}\oint\nolimits_{C_{j}%
}e^{\lambda_{j}\left(  z\right)  +t\,z}dz,\label{f5_2a3}%
\end{equation}
and as a consequence of (\ref{f5_2a2}) and (\ref{f5_2a3}) we have%
\begin{equation}
\sum_{b_{j}<t}\frac{1}{2\pi i}\oint\nolimits_{C_{j}}e^{\lambda_{j}%
(z)+t\,z}dz\,>\,0.\label{f5_2a4}%
\end{equation}
The definition of the objects $\pi_{\lambda}$, $\lambda$, $\lambda_{j}$,
$C_{j}$, $j=1,\ldots,I$, in (\ref{f5_2a1}) through (\ref{f5_2a4}) were listed
in (ii), (i), (iii) and (iv) in the last subsection.$\medskip$
\end{proposition}

The proof of Proposition \ref{sec5_prop1} will be prepared by two lemmas and
several technical definitions. Throughout the present subsection the numbers
$t\in(b_{I},b_{I+1})$ and $I\in\{1,\ldots,n-1\}$ are kept fixed, and
Assumption 3 is effective.

We define%
\begin{align}
& D_{\pm}\,:=\,\{\,\zeta\in\mathcal{R}_{\lambda}\,|\,\pm\operatorname*{Im}%
(\pi_{\lambda}(\zeta))>0,\text{ }\pm\operatorname*{Im}(\lambda(\zeta
)+t\,\pi_{\lambda}(\zeta))>0\,\},\bigskip\nonumber\\
& D\,:=\,\operatorname*{Int}\left(  \overline{D_{+}\cup D_{-}}\right)
.\label{f5_2b1}%
\end{align}
The set $D\subset\mathcal{R}_{\lambda}$ is open, but not necessarily
connected. Since the algebraic function $\lambda$ is of real type, we have
$\varrho(D_{\pm})=D_{\mp}$ and $D_{\pm}\subset\mathcal{R}_{\pm}$\ with the
reflection function $\varrho$ and Riemann surfaces $\mathcal{R}_{+}$ and
$\mathcal{R}_{-}$\ from (v) in the listing in the last subsection.

By $Cr\subset\mathcal{R}_{\lambda}$ we denote the set of critical points of
the function $\operatorname*{Im}(\lambda+t\,\pi_{\lambda})$, which are at the
same time the critical points of $\operatorname{Re}(\lambda+t\,\pi_{\lambda}%
)$, and the zeros of the derivative $(\lambda+t\,\pi_{\lambda})^{\prime}$.
Since $\mathcal{R}_{\lambda}$ is compact, it follows that $Cr$ is
finite.$\medskip$

\begin{lemma}
\label{sec5_Lem1}(i) The boundary $\partial D\subset\mathcal{R}_{\lambda}$
consists of a chain%
\begin{equation}
\gamma=\gamma_{1}+\cdots+\gamma_{K}\label{f5_2c1}%
\end{equation}
of $K$ piecewise analytic Jordan curves $\gamma_{k}$, $k=1,\ldots,K$. The
orientation of each $\gamma_{k}$, $k=1,\ldots,K,$ is chosen in such a way that
the domain $D$ lies to its left. The curves $\gamma_{k}$, $k=1,\ldots,K$, are
not necessarily disjoint, however, intersections are possible only at critical
points $\zeta\in Cr$.$\medskip$

(ii) The choice of the Jordan curves $\gamma_{k}$, $k=1,\ldots,K$, in
(\ref{f5_2c1})\ can be done in such a way that each of them is invariant under
the reflection function $\varrho$ except for its orientation, i.e., we have
$\varrho(\gamma_{k})=-\gamma_{k}$ for $k=1,\ldots,K$.$\medskip$

(iii) Let $2s_{k}$ be the length of the Jordan curve $\gamma_{k}$,
$k=1,\ldots,K$; with a parameterization by arc length we then have $\gamma
_{k}:\left[  0,2s_{k}\right]  \longrightarrow\partial D\subset\mathcal{R}%
_{\lambda}$. The starting point $\gamma_{k}(0)$ can be chosen in such a way
that%
\begin{equation}
\gamma_{k}(\left(  0,s_{k}\right)  )\subset\partial D_{+}\diagdown\pi
_{\lambda}^{-1}(\mathbb{R})\text{ \ and \ }\gamma_{k}(\left(  s_{k}%
,2s_{k}\right)  )\subset\partial D_{-}\diagdown\pi_{\lambda}^{-1}%
(\mathbb{R}).\smallskip\label{f5_2c2}%
\end{equation}

(iv) The function $\operatorname{Re}\left(  \lambda\circ\gamma_{k}+t\,\left(
\pi_{\lambda}\circ\gamma_{k}\right)  \right)  $ is monotonically increasing on
$(0,\allowbreak s_{k})$, mo\-no\-tonically decreasing on $\left(  s_{k}%
,2s_{k}\right)  $, and these monotonicities are strict at each $\zeta\in
\gamma_{k}\diagdown(Cr\cup\pi_{\lambda}^{-1}(\mathbb{R}))$.$\medskip$
\end{lemma}

\begin{proof}
The function $\operatorname*{Im}(\lambda+t\,\pi_{\lambda})$ is harmonic in
$\mathcal{R}_{\lambda}\diagdown\allowbreak\pi_{\lambda}^{-1}(\{\infty\}) $. As
a system of level lines of an harmonic function, $\partial D$ consists of
piecewise analytic arcs, and their orientation can be chosen in such a way
that the domain $D$ lies to the left of $\partial D$. Since $\partial
D\diagdown Cr$ consists of analytic arcs, locally each $\zeta\in\partial
D\diagdown Cr$ touches only two components of $\mathcal{R}_{\lambda}%
\diagdown\partial D$, and locally it belongs only to one of the analytic
Jordan subarcs of $\partial D\diagdown Cr$. Globally, for each $\zeta
\in\partial D$ there exists at least one Jordan curve $\widetilde{\gamma}$ in
$\partial D$ with $\zeta\in\widetilde{\gamma}$, but this association is in
general not unique, different choices may be possible, and the cuts that are
candidates for such a choice bifurcate only at points in $Cr$. By a stepwise
exhaustion it follows that $\partial D$ is the union of Jordan curves, i.e.,
we have%
\begin{equation}
\partial D\,=\,\gamma\,=\,\gamma_{1}+\gamma_{2}+\cdots\label{f5_2c1b}%
\end{equation}
Different curves $\gamma_{k}$ may intersect, but because of the Implicit
Function Theorem, intersections are possible only at points in $Cr$.

After these considerations it remains only to show in assertion (i) that the
number of Jordan curves $\gamma_{k}$ in (\ref{f5_2c1b}) is finite; basically
this follows from the compactness of $\mathcal{R}_{\lambda}$. If we assume
that there exist infinitely many curves $\gamma_{k}$ in (\ref{f5_2c1b}), then
there exists at least one cluster point $z^{\ast}\in\mathcal{R}_{\lambda}$
such that any neighborhood of $z^{\ast}$ intersects infinitely many curves
$\gamma_{k}$ from (\ref{f5_2c1b}). Obviously, $z^{\ast}\in\pi_{\lambda}%
^{-1}(\{\infty\})$ is impossible. Let $z:V\longrightarrow\mathbb{D}$ be a
local coordinate of $z^{\ast}$ that maps a neighborhood $V$ of $z^{\ast}%
$\ conformally onto the unit disk $\mathbb{D}$ with $z(z^{\ast})=0$. The
function $g:=\operatorname*{Im}(\lambda+t\,\pi_{\lambda})\circ z^{-1}$ is
harmonic in $\mathbb{D}$ and not identically constant. If $g$ has a critical
point of order $m$ at the origin, then, because of the local structure of
level lines near a critical point, small neighborhoods of the origin can
intersect only with at most $m$ elements of the set $\left\{  \,z(\gamma
_{k}|_{V});\,k=1,2,\ldots\,\right\}  $. If, on the other hand, $g$ has no
critical point at the origin, then it follows from the Implicit Function
Theorem that small neighborhoods of the origin can intersect with at most one
element of the set $\left\{  \,z(\gamma_{k}|_{V});\,k=1,2,\ldots\,\right\}  $.
Hence, the assumption that $z^{\ast} $ is a cluster point of curves
$\gamma_{k}$ from (\ref{f5_2c1b}) is impossible, and the finiteness of the sum
in (\ref{f5_2c1b}) is proved, which completes the proof of assertion
(i).$\smallskip$

For each Jordan curve $\gamma_{k},$ $k=1,\ldots,K,$ in (\ref{f5_2c1}) we
deduce from (\ref{f5_2b1})\ that%
\begin{equation}
\frac{\partial}{\partial n}\operatorname*{Im}(\lambda(\zeta)+t\,\pi_{\lambda
}(\zeta))\,>\,0\text{ \ for each \ }\zeta\in\gamma_{k}\cap(\mathcal{R}%
_{+}\diagdown Cr),\label{f5_2c5}%
\end{equation}
and since the orientation of $\partial D=\gamma$ has been chosen such that $D$
lies to the left of each $\gamma_{k}$, we further have%
\begin{equation}
\frac{\partial}{\partial t}\operatorname{Re}(\lambda(\zeta)+t\,\pi_{\lambda
}(\zeta))\,>\,0\text{ \ for each \ }\zeta\in\gamma_{k}\cap(\mathcal{R}%
_{+}\diagdown Cr)\label{f5_2c4}%
\end{equation}
by the Cauchy-Riemann differential equations. In (\ref{f5_2c5}),
$\partial/\partial n$ denotes the normal derivative on $\gamma_{k}$ pointing
into $D$, and in (\ref{f5_2c4}), $\partial/\partial t$ denotes the tangential
derivative. In $\mathcal{R}_{-}$, we get the corresponding inequality%
\begin{equation}
\frac{\partial}{\partial t}\operatorname{Re}(\lambda(\zeta)+t\,\pi_{\lambda
}(\zeta))\,<\,0\text{ \ for each \ }\zeta\in\gamma_{k}\cap(\mathcal{R}%
_{-}\diagdown Cr).\label{f5_2c6}%
\end{equation}

Since $\lambda$ is a function of real type, we deduce with the help of the
reflection function $\varrho$ that%
\[
\left(  \lambda\circ\varrho\right)  (\zeta)+t\,\left(  \pi_{\lambda}%
\circ\varrho\right)  (\zeta)=\overline{\lambda(\zeta)+t\,\pi_{\lambda}(\zeta
)}\text{ \ \ for \ \ }\zeta\in\mathcal{R}_{\lambda},
\]
and therefore also that%
\begin{equation}
\varrho(\partial D)=\partial D.\label{f5_2c3}%
\end{equation}

As a first consequence of (\ref{f5_2c4}) and (\ref{f5_2c6}) we conclude that
none of the Jordan curves $\gamma_{k}$ in (\ref{f5_2c1}) can be contained
completely in $\overline{\mathcal{R}}_{+}$ or $\overline{\mathcal{R}}_{-}$.
Indeed, if we assume that some $\gamma_{k}$ is contained in $\overline
{\mathcal{R}}_{+}$, then it would follow from (\ref{f5_2c4}) that
$\operatorname{Re}(\lambda+t\,\pi_{\lambda})$ could not be continues along the
whole curve $\gamma_{k}$.

Since each $\gamma_{k},$ $k=1,\ldots,K,$ in (\ref{f5_2c1}) intersects at the
same time $\mathcal{R}_{+}$ and $\mathcal{R}_{-}$, it follows that all curves
$\gamma_{k}$ can be chosen from $\partial D$ in the exhaustion process in the
proof of assertion (i) in such a way that $\varrho(\gamma_{k})=-\gamma_{k}$
for each $k=1,\ldots,K$, which proves assertion (ii). We remark that a choice
between different options for a selection of the $\gamma_{k},$ $k=1,\ldots,K,$
exists only if points of the intersection $\gamma_{k}\cap\pi_{\lambda}%
^{-1}(\mathbb{R})$ belong to $Cr$.$\smallskip$

From the fact that each $\gamma_{k}$ in (\ref{f5_2c1}) is a Jordan curve,
which is neither fully contained in $\overline{\mathcal{R}}_{+}$ nor in
$\overline{\mathcal{R}}_{-}$ and that we have $\varrho(\gamma_{k})=-\gamma
_{k}$, we deduce that $\gamma_{k}\cap\pi_{\lambda}^{-1}(\mathbb{R})$ consists
of exactly two points. By an appropriate choice of the starting point of the
parameterization of $\gamma_{k}$ in $\gamma_{k}\cap\pi_{\lambda}%
^{-1}(\mathbb{R})$ it follows that (\ref{f5_2c2}) is satisfied, which proves
assertion (iii).$\smallskip$

The monotonicity statements in assertion (iv) are immediate consequences of
(\ref{f5_2c4}) and (\ref{f5_2c6}), which completes the proof of Lemma
\ref{sec5_Lem1}.$\medskip$
\end{proof}

\begin{lemma}
\label{sec5_Lem2}We have%
\begin{equation}
\frac{1}{2\pi i}\oint\nolimits_{\gamma_{k}}e^{\lambda(\zeta)+t\,\pi_{\lambda
}(\zeta)}d\zeta\,<\,0\ \ \ \text{for each }\ k=1,\ldots,K.\smallskip
\label{f5_2d1}%
\end{equation}

\end{lemma}

\begin{proof}
We abbreviate the integrand in (\ref{f5_2d1}) by%
\[
g(\zeta):=e^{\lambda(\zeta)+t\,\pi_{\lambda}(\zeta)},\text{ \ \ }\zeta
\in\mathcal{R}_{\lambda}\diagdown\pi_{\lambda}^{-1}(\{\infty\}),
\]
and assume $k\in\{1,\ldots,K\}$ in (\ref{f5_2d1}) to be fixed.

From assertion (i) in Lemma \ref{sec5_Lem1} we know that $\operatorname*{Im}%
g(\zeta)=0$ for all $\zeta\in\gamma_{k}$, from assertion (iv) we further know
that $\operatorname{Re}g(\zeta)=g(\zeta)$ is strictly increasing on
$\gamma_{k}\cap(\mathcal{R}_{+}\diagdown Cr)$, from (\ref{f5_2c2}) that
$\gamma_{k}\cap\,\mathcal{R}_{+}$ is the subarc $\left.  \gamma_{k}\right\vert
_{(0,s_{k})}$,\ and from the proof of assertion (iv) it is evident that also
the following slightly stronger statement%
\begin{equation}
(g\circ\gamma_{k})^{\prime}(s)>0\text{ \ \ for \ \ }0<s<s_{k}\text{ \ and
\ }\gamma_{k}(s)\notin Cr\label{f5_2d2}%
\end{equation}
holds true. It further follows from (\ref{f5_2c2}) that we have%
\begin{equation}
\operatorname*{Im}\pi_{\lambda}\circ\gamma_{k}(0)=\operatorname*{Im}%
\pi_{\lambda}\circ\gamma_{k}(s_{k})=0\text{ \ and \ }\operatorname*{Im}%
\pi_{\lambda}\circ\gamma_{k}(s)>0\text{ \ for \ }0<s<s_{k}.\label{f5_2d3}%
\end{equation}
Let the coordinates $z,x,y$ and the differentials $dz,dx,dy$\ be defined by
$\pi_{\lambda}(\zeta)=z=x+iy\in\mathbb{C}$, $\zeta\in\gamma_{k}$, and
$dz=dx+idy$, and let these coordinates and differentials be lifted from
$\overline{\mathbb{C}}$ onto $\mathcal{R}_{\lambda}$, where we then have
$\zeta=\xi+i\,\eta$ and $d\zeta=d\xi+i\,d\eta$. Taking into consideration that
$\varrho(\gamma_{k})=-\gamma_{k}$, $\varrho(d\zeta)=\overline{d\zeta}$, and
$(g\circ\varrho)(\zeta)=\overline{g(\zeta)}=g(\zeta)$ for all $\zeta\in
\gamma_{k}$, we conclude that%
\begin{align}
& \frac{1}{2\pi i}\oint\nolimits_{\gamma_{k}}g(\zeta)d\zeta=\frac{1}{2\pi
i}\int_{\gamma_{k}\cap D_{+}}\cdots+\frac{1}{2\pi i}\int_{\gamma_{k}\cap
D_{-}}g(\zeta)\left(  d\xi+i\,d\eta\right)  \medskip\nonumber\\
& \text{\ \ \ \ \ \ \ }=\frac{1}{\pi}\int_{\gamma_{k}\cap D_{+}}g(\zeta
)d\eta\,=\,\frac{1}{\pi}\int_{0}^{s_{k}}(g\circ\gamma_{k}%
)(s)\operatorname*{Im}\left(  (\pi_{\lambda}\circ\gamma_{k})^{\prime
}(s)\right)  ds\medskip\nonumber\\
& \text{ \ \ \ \ \ }=-\frac{1}{\pi}\int_{0}^{s_{k}}(g\circ\gamma_{k})^{\prime
}(s)\operatorname*{Im}\left(  \pi_{\lambda}\circ\gamma_{k}(s)\right)  ds\text{
}<\text{\ }0.\label{f5_2d4}%
\end{align}
Indeed, the first three equalities in (\ref{f5_2d4}) are a consequence of the
specific symmetries and antisymmetries with respect to $\varrho$ that have
been listed just before (\ref{f5_2d4}). From the three equalities we consider
the second one in more detail, and concentrate on the transformation of the
second integral after the first equality. We have%
\begin{align*}
\frac{1}{2\pi i}\int_{\gamma_{k}\cap D_{-}}g(\zeta)\left(  d\xi+i\,d\eta
\right)  \,  & =\,\frac{-1}{2\pi i}\int_{\gamma_{k}\cap D_{+}}g(\zeta)\left(
d\xi-i\,d\eta\right)  \medskip\\
& =\,\frac{1}{2\pi i}\int_{\gamma_{k}\cap D_{+}}g(\zeta)\left(  -d\xi
+i\,d\eta\right)  ,
\end{align*}
which verifies the second equality. The last equality in (\ref{f5_2d4})
follows from integration by parts together with the equalities in
(\ref{f5_2d3}). The inequality in (\ref{f5_2d4}) is then a consequence of
(\ref{f5_2d2}) and the inequality in (\ref{f5_2d3}).$\medskip$
\end{proof}

\begin{proof}
[Proof of Proposition \ref{sec5_prop1}]The chain $\gamma$ of oriented Jordan
curves (\ref{f5_2c1}) in Lemma \ref{sec5_Lem1} is the candidate for the chain
$\gamma$ in Proposition \ref{sec5_prop1}. Equality (\ref{f5_2a1}) and
inequality (\ref{f5_2a2}) have been verified by the Lemmas \ref{sec5_Lem1} and
\ref{sec5_Lem2}, respectively. Identity (\ref{f5_2a3}) and its consequence
(\ref{f5_2a4}) remain to be proved.

As integration paths $C_{j}$, $j=1,\ldots,I$, on the right-hand side of
(\ref{f5_2a3}) we take the common Jordan curve $C$ from (iv) in the listing in
the last subsection. The set $\pi_{\lambda}^{-1}(\overline{\operatorname*{Ext}%
(C)})$ consists of $n$ disjoint components if $C$ is chosen sufficiently close
to infinity; it then also follows that all branch points of $\lambda$ are
contained in $\mathcal{R}_{\lambda}\diagdown\allowbreak\pi_{\lambda}%
^{-1}(\overline{\operatorname*{Ext}(C)})$. Further, we have%
\begin{equation}
\operatorname*{Im}(\lambda_{j}(z)+t\,z)\left\{
\begin{array}
[c]{ll}%
\,>\,0 & \text{ \ for all \ }z\in C,\text{ }\operatorname*{Im}(z)>0\\
\,<\,0 & \text{ \ for all \ }z\in C,\text{ }\operatorname*{Im}(z)<0
\end{array}
\right.  ,\text{ }j=1,\ldots,I,\label{f5_2e1}%
\end{equation}
and%
\begin{equation}
\operatorname*{Im}(\lambda_{j}(z)+t\,z)\left\{
\begin{array}
[c]{ll}%
\,<\,0 & \text{ \ for all \ }z\in C,\text{ }\operatorname*{Im}(z)>0\\
\,>\,0 & \text{ \ for all \ }z\in C,\text{ }\operatorname*{Im}(z)<0
\end{array}
\right.  ,\text{ }j=I+1,\ldots,n.\label{f5_2e2}%
\end{equation}
A choice of $C$\ with these properties is possible because of (\ref{f2_1a1})
in Lemma \ref{sec2_Lem1} in Subsection \ref{sec2_1} and the assumption that
$b_{1}\leq\,\cdots\,\leq b_{I}<t<b_{I+1}\leq\,\cdots\,\leq\,b_{n}$.

Next we define%
\begin{equation}
D_{0}:=D\diagdown\pi_{\lambda}^{-1}(\overline{\operatorname*{Ext}(C)}%
)\subset\mathcal{R}_{\lambda}.\label{f5_2e3}%
\end{equation}
From (\ref{f5_2e1}), (\ref{f5_2e2}), and (\ref{f5_2b1}) it follows that
exactly $I$ of the $n$ components $\widehat{C}_{j}\subset\mathcal{R}_{\lambda
}$, $j=1,\ldots,n$, of $\pi_{\lambda}^{-1}(\overline{\operatorname*{Ext}(C)})$
are contained in $D$. Each $\widehat{C}_{j}$ lies in a different sheet
$S_{\lambda}^{(j)}$, $j=1,\ldots,n,$ of the system of standard sheets
introduced in Lemma \ref{sec2_Lem4} in Subsection \ref{sec2_2}. The
enumeration of the sheets $S_{\lambda}^{(j)}$ corresponds to that of the
functions $\lambda_{j}$ as stated in (\ref{f2_2b1}). Let $\widetilde{C}%
_{j}\subset\mathcal{R}_{\lambda}$, $j=1,\ldots,n$, denote the lifting of the
oriented Jordan curve $C\subset\mathbb{C}$ onto $S_{\lambda}^{(j)}%
\subset\mathcal{R}_{\lambda}$. We then have $\pi_{\lambda}(\widetilde{C}%
_{j})=C_{j}=C$ for $j=1,\ldots,n$, and from (\ref{f2_2b1}) it follows that%
\begin{equation}
\lambda(\zeta)\,=\,\lambda_{j}(\pi_{\lambda}(\zeta))\text{ \ \ for \ \ }%
\zeta\in\widetilde{C}_{j}\text{, \ }j=1,\ldots,n\text{.}\label{f5_2e6}%
\end{equation}

Since $\widetilde{C}_{j}=\partial\widehat{C}_{j}$ for $j=1,\ldots,n$, the open
set $D_{0}$ lies to the left of each $\widetilde{C}_{j}$. Together with
assertion (i) of Lemma \ref{sec5_Lem1}, it follows from (\ref{f5_2e3}) that
the chain%
\begin{equation}
\gamma+\widetilde{C}_{1}+\cdots+\widetilde{C}_{I}\,=\,\gamma_{1}+\cdots
+\gamma_{K}+\widetilde{C}_{1}+\cdots+\widetilde{C}_{I}\subset\mathcal{R}%
_{\lambda}\label{f5_2e4}%
\end{equation}
forms the contour $\partial D_{0}$ with an orientation for which $D_{0}$ lies
everywhere to its left. By Cauchy's Theorem we have%
\begin{equation}
\frac{1}{2\pi i}\oint\nolimits_{\gamma+\widetilde{C}_{1}+\cdots+\widetilde
{C}_{I}}e^{\lambda(\zeta)+t\,\pi_{\lambda}(\zeta)}d\zeta=0.\label{f5_2e5}%
\end{equation}
Identity (\ref{f5_2a3}) follows immediately from (\ref{f5_2e5}) and
(\ref{f5_2e6}). Inequality (\ref{f5_2a4}) is a consequence of (\ref{f5_2a2})
and (\ref{f5_2a3}) since we have%
\begin{align}
\sum_{b_{j}\,<\,t}\frac{1}{2\pi i}\oint\nolimits_{C_{j}}e^{\lambda_{j}\left(
\zeta\right)  +t\,\zeta}d\zeta & =\sum_{j=1}^{I}\frac{1}{2\pi i}%
\oint\nolimits_{C_{j}}e^{\lambda_{j}\left(  \zeta\right)  +t\,\zeta}%
d\zeta\nonumber\\
& =\frac{-1}{2\pi i}\oint\nolimits_{\gamma}e^{\lambda(\zeta)+t\,\pi_{\lambda
}(\zeta)}d\zeta\,>\,0.\label{f5_2e7}%
\end{align}

\end{proof}

\subsection{\label{sec5_3}A Preliminary Proof of Positivity}

\qquad With Proposition \ref{sec5_prop1} we are prepared for the proof of
positivity of the measure $\mu_{A,B}$ in Theorems \ref{sec1_thm2} under
Assumption 3, which then completes the proof of Theorems \ref{sec1_thm2} under
Assumption 3.

\begin{proof}
[Proof of Positivity under Assumption 3]From representation (\ref{f1_3a}) in
Theorem \ref{sec1_thm2} it is obvious that the discrete part%
\begin{equation}
d\mu_{d}\,=\,\sum_{j=1}^{n}e^{\widetilde{a}_{jj}}\delta_{\widetilde{b}_{j}%
}\,=\,\sum_{j=1}^{n}e^{a_{jj}}\delta_{b_{j}}\label{f5_3b2}%
\end{equation}
of the measure $\mu_{A,B}$\ is positive. From (\ref{sec5_4}) of Proposition
\ref{sec5_prop1} it follows that the density function $w_{A,B}$ in
(\ref{f1_3b}) of Theorem \ref{sec1_thm2} is positive on $\left[  \widetilde
{b}_{1},\widetilde{b}_{n}\right]  \diagdown\{\widetilde{b}_{1},\ldots
,\widetilde{b}_{n}\}=\left[  b_{1},b_{n}\right]  \diagdown\{b_{1},\ldots
,b_{n}\}$, which proves the positivity of the measure $\mu_{A,B}$. Notice that
the last identity holds because of Assumption 1 in Section \ref{sec1b}.
\end{proof}

Under Assumption 3, relation (\ref{f1_3d}) in Theorem \ref{sec1_thm2} is
proved in a slightly stronger form.

\begin{lemma}
\label{sec5_Lem3}Under Assumption 3 we have%
\begin{equation}
\operatorname*{supp}\left(  \mu_{A,B}\right)  \,=\,\left[  b_{1},b_{n}\right]
\,=\,\left[  \widetilde{b}_{1},\widetilde{b}_{n}\right]  .\label{f5_3b1}%
\end{equation}

\end{lemma}

\begin{proof}
The lemma is an immediate consequence of the strict inequality in
(\ref{f5_2a4}) in Proposition \ref{sec5_prop1}.$\medskip$
\end{proof}

\subsection{\label{sec5_4}The General Case}

\qquad In the present subsection we show that Assumption 3, which has played a
central role in the last subsection, is actually superfluous for proof of
positivity of the measure $\mu_{A,B}$ in Theorems \ref{sec1_thm2}. For this
purpose we have to revisit some definitions and results from Subsections
\ref{sec2_1} and \ref{sec2_2}.

If the polynomial $g(\lambda,t)$\ in (\ref{f2_1a2}) is not irreducible, then
it can be factorized into $m>1$ irreducible factors $g_{(l)}(\lambda,t) $,
$l=1,\ldots,m$, of degree $n_{l}$ as already stated in (\ref{f2_1a2}). For the
partial degrees $n_{l}$ we have $n_{1}+\cdots+n_{m}=n$. Each polynomial
$g_{(l)}(\lambda,t)$, $l=1,\ldots,m$, can be normalized in accordance to
(\ref{f2_1a4}).

The $m$ polynomial equations (\ref{f2_1a3}) define $m$ algebraic functions
$\lambda_{(l)}$, $l=1,\ldots,\allowbreak m$, and each of them has a Riemann
surface $\mathcal{R}_{\lambda,l}$, $l=1,\ldots,m$, with $n_{l}$ sheets over
$\overline{\mathbb{C}}$ as its natural domain of the definition. The solution
$\lambda$ of equation (\ref{f2_1a2}) consists of these $m$ algebraic
functions, and its domain of definition is the union (\ref{f2_2a1}) of the $m$
Riemann surfaces $\mathcal{R}_{\lambda,l}$, $l=1,\ldots,m$.

Each algebraic function $\lambda_{(l)}$, $l=1,\ldots,m$, possesses $n_{l}$
branches $\lambda_{l,i}$, $i=1,\ldots,n_{l}$, which are assumed to be chosen
analogously to Definition \ref{sec2_def2} in Subsection \ref{sec2_2}, but with
a new form of indices. After (\ref{f2_1a4}) we have denoted by $j:\{\,(l,i)$,
$i=1,\ldots,n_{l}$, $l=1,\ldots,m\,\}\,\longrightarrow\,\{\,1,\ldots,n\,\}$ a
bijection that establishes a one-to-one correspondence between the two types
of indices that are relevant here. We can assume that this correspondence has
been chosen in such a way that%
\begin{equation}
b_{j(l,1)}\,\leq\cdots\leq\,b_{j(l,n_{l})}\text{ \ \ for each \ \ }%
l=1,\ldots,m,\label{f5_3c1}%
\end{equation}
and in the new system of indices (\ref{f2_1a1}) in Lemma \ref{sec2_Lem1} takes
the form%
\begin{equation}
\lambda_{j(l,i)}(t)\,=\,\lambda_{l,i}(t)\,=\,a_{j(l,i),j(l,i)}-b_{j(l,i)}%
t+\text{O}\left(  1/t\right)  \text{ \ \ as \ \ }t\rightarrow\infty
\label{f5_3c2}%
\end{equation}
for $i=1,\ldots,n_{l}$, $l=1,\ldots,m$.

We define%
\begin{equation}
w_{A,B,l}(t)\,:=\,\sum_{i=1,\text{ }b_{j(l,i)}\,<\,t}^{n_{l}}\frac{1}{2\pi
i}\oint\nolimits_{C_{l,i}}e^{\lambda_{l,i}(\zeta)+t\,\zeta}d\zeta\text{
\ for\ \ }l=1,\ldots,m\label{f5_3c3}%
\end{equation}
with $C_{l,i}=C_{j(l,i)}$. From (\ref{f5_3c3}) it follows that in
(\ref{f1_3b}) and (\ref{f1_3c}) in Theorem \ref{sec1_thm2} we have%
\begin{equation}
w_{A,B}(t)\,=\,\sum_{l=1}^{m}w_{A,B,l}(t).\label{f5_3c4}%
\end{equation}
Under\ Assumption 3 the new definitions remain consistent in a trivial way
with $m=1$.

In the general proof of positivity of the measure $\mu_{A,B}$ the next
proposition will take the role of Proposition \ref{sec5_prop1}.$\smallskip$

\begin{proposition}
\label{sec5_prop2}(i) \ For each $l\in\{\,1,\ldots,m\,\}$ with $n_{l}=1$ we
have%
\begin{equation}
w_{A,B,l}(t)\,=\,0\text{ \ \ for all \ \ }t\in\mathbb{R}_{+}.\label{f5_3d1}%
\end{equation}
(ii) \ For each $l\in\{\,1,\ldots,m\,\}$ with $n_{l}>1$ we have%
\begin{equation}
w_{A,B,l}(t)\,\left\{
\begin{array}
[c]{ccl}%
>\,0 & \text{ \ for all \ } & t\in\left[  b_{j(l,1)},b_{j(l,n_{l})}\right]
\diagdown\{b_{j(l,1)},\ldots,b_{j(l,n_{l})}\}\medskip\\
=\,0 & \text{ \ for all\ \ } & t\in\mathbb{R}_{+}\diagdown\left[
b_{j(l,1)},b_{j(l,n_{l})}\right]  .
\end{array}
\right. \label{f5_3d2}%
\end{equation}
Each function $w_{A,B,l}$, $l=1,\ldots,m$, is the restriction of an entire
function in each interval of $\left[  b_{j(l,1)},b_{j(l,n_{l})}\right]
\diagdown\{b_{j(l,1)},\ldots,b_{j(l,n_{l})}\}$.$\smallskip$
\end{proposition}

\begin{proof}
Equality (\ref{f5_3d1}) and the equality in the second line of (\ref{f5_3d2})
follow from (\ref{f5_3c3}) and the analogue of Lemma \ref{sec3_lem1} in
Subsection \ref{sec3_1}, which also holds for each complete set of branches
$\lambda_{l,i}$, $i=1,\ldots,n_{l}$, of the algebraic function $\lambda_{(l)}%
$, $l=1,\ldots,m$. In case of the second line in (\ref{f5_3d2}) we have also
to take in consideration the ordering (\ref{f5_3c1}).

For the proof of the inequality in the first line of (\ref{f5_3d2}) we have to
redo the analysis in the proofs of Lemmas \ref{sec5_Lem1}, \ref{sec5_Lem2},
and of Proposition \ref{sec5_prop1}, but now with the role of algebraic
function $\lambda$, the Riemann surface $\mathcal{R}_{\lambda}$, and the
branches $\lambda_{j}$, $j=1,\ldots,n$, taken over by $\lambda_{(l)}$,
$\mathcal{R}_{\lambda,l}$, and $\lambda_{l,i}$, $i=1,\ldots,n_{l}$,
respectively, for each $l=1,\ldots,m$ with $n_{l}>1$. It is not difficult to
see that this transition is a one-to-one copying of all steps of the earlier
analysis, and we will not go into further details. The inequality in the first
line of (\ref{f5_3d2}) follows then together with (\ref{f5_3c3}) as an
analogue of (\ref{f5_2a4}) in Proposition \ref{sec5_prop1}.

It follows from (\ref{f5_3c3}) that each $w_{A,B,l}$ is the restriction of an
entire function in each interval in $\left[  b_{j(l,1)},b_{j(l,n_{l})}\right]
\diagdown\{b_{j(l,1)},\ldots,b_{j(l,n_{l})}\}$ for $l=1,\ldots,m$.
\end{proof}

\subsection{\label{sec5_5}General Proof of Positivity}

\qquad With (\ref{f5_3c3}) and Proposition \ref{sec5_prop2} we are prepared
for the proof of positivity without Assumption 3.

\begin{proof}
[General Proof of Positivity]Since the discrete part (\ref{f5_3b2}) of the
measure $\mu_{A,B}$\ is positive, it remains only to show that the density
function $w_{A,B}$ in (\ref{f1_3b}) of Theorem \ref{sec1_thm2} is non-negative
in $\left[  \widetilde{b}_{1},\widetilde{b}_{n}\right]  \diagdown
\{\widetilde{b}_{1},\ldots,\widetilde{b}_{n}\}=\left[  b_{1},b_{n}\right]
\diagdown\{b_{1},\ldots,\allowbreak b_{n}\}$. But this follows immediately
from (\ref{f5_3d2}) and (\ref{f5_3d1}) in Proposition \ref{sec5_prop2}
together with (\ref{f5_3c3}). Notice that because of Assumption 1 in Section
\ref{sec1b} we have $\widetilde{b}_{j}=b_{j}$ for $j=1,\ldots,n$.
\end{proof}

\section{\label{sec5b}Summing up the Proofs of Theorems \ref{sec1_thm1} and
\ref{sec1_thm2}}

All assertions of Theorem \ref{sec1_thm2}, except for the positivity of the
measure $\mu_{A,B}$, have been proved in Section \ref{sec3}, and after the
proof of positivity in the last section, the proof of Theorem \ref{sec1_thm2}
is complete.

Theorem \ref{sec1_thm1} is an immediate consequence of Theorem \ref{sec1_thm2}.

\section{\label{sec6}Proof of Proposition \ref{sec1_prop1}}

\qquad The proof of Proposition \ref{sec1_prop1} is given in two steps. In the
first one, the formulae (\ref{f1_1d}) and (\ref{f1_1e}) are verified. After
that in Subsection \ref{sec6_2}, it is shown that the density function
$w_{A,B}(x)$ in (\ref{f1_1e}) is positive for $b_{1}<x<b_{2}$. In the last
subsection, representation (\ref{f1_1e}) of the density function $w_{A,B}$ in
Proposition \ref{sec1_prop1} is compared with the corresponding result in
\cite{MehtaKumar76}.\smallskip

\subsection{Proof of the Representations (\ref{f1_1d}) and (\ref{f1_1e})}

\qquad Representation (\ref{f1_1d}) of the general structure of the measure
$\mu_{A,B}$ follows as a special case from the analogous result (\ref{f1_3a})
in Theorem \ref{sec1_thm2}. From (\ref{f1_3b}) we further deduce that the
density function $w_{A,B}$ in (\ref{f1_1d}) can be represented as%
\begin{equation}
w_{A,B}(x)=\frac{1}{2\pi i}\oint\nolimits_{C_{1}}e^{\lambda_{1}(\zeta
)+x\,\zeta}d\zeta\text{ \ \ \ for \ \ }b_{1}<x<b_{2}\label{f6_1a1}%
\end{equation}
with $\lambda_{1}$ the branch of the algebraic function $\lambda$ of degree
$2$ defined by the polynomial equation%
\begin{align}
g(\lambda,t)  & =\det\left(  \lambda\,I-(A-t\,B)\right)  \medskip\nonumber\\
& =(\lambda+b_{1}t-a_{11})(\lambda+b_{2}t-a_{22})-|a_{12}|^{2}=0\label{f6_1a2}%
\end{align}
that satisfies%
\begin{equation}
\lambda_{1}(t)=a_{11}-b_{1}t\,+\,O(t^{-1})\text{ \ \ as \ \ }t\rightarrow
\infty.\label{f6_1a3}%
\end{equation}
Further, the integration path $C_{1}$ in (\ref{f6_1a1}) is\ a positively
oriented Jordan curve that contains all branch points of the function
$\lambda$ in its interior. From (\ref{f6_1a2}) and (\ref{f6_1a3}) it follows
that $\lambda_{1}$ is explicitly given by%
\begin{equation}
\lambda_{1}(t)=\frac{1}{2}\left[  (a_{22}+a_{11})-(b_{2}+b_{1})\,t+\sqrt
{\left[  (a_{11}-a_{22})+(b_{2}-b_{1})\,t\right]  ^{2}+4\,|a_{12}|^{2}}\right]
\label{f6_1a4}%
\end{equation}
with the sign of the square root in (\ref{f6_1a4}) chosen in such a way that
$\sqrt{\cdots}\approx(b_{2}-b_{1})\,t$ for $t$\ near $\infty$. Evidently,
$\lambda_{1}$ has the two branch points%
\begin{equation}
t_{1,2}=\frac{a_{22}-a_{11}}{b_{2}-b_{1}}\pm i\,\frac{2\,|a_{12}|}{b_{2}%
-b_{1}}.\label{f6_1b1}%
\end{equation}

The main task is now to transform the right-hand side of (\ref{f6_1a1}) into
the more explicit expression in (\ref{f1_1e}). In order to simplify the
exponent in (\ref{f6_1a1}), we introduce a new variable $v$ by the
substitution%
\begin{equation}
t(v):=\frac{a_{22}-a_{11}}{b_{2}-b_{1}}+\frac{2}{b_{2}-b_{1}}\,v,\text{
\ \ }v\in\mathbb{C},\label{f6_1b2}%
\end{equation}
which leads to%
\begin{align}
& \left(  \lambda_{1}\circ t\right)  (v)+x\,t(v)\medskip\nonumber\\
& \text{ \ \ \ \ \ \ }=\frac{a_{11}(b_{2}-x)+a_{22}(x-b_{1})}{b_{2}-b_{1}%
}\,+\,\frac{2\,x-(b_{2}+b_{1})}{b_{2}-b_{1}}\,v\,+\,\sqrt{|a_{12}|^{2}+v^{2}%
}\text{ \ \ \ \ \ \ }\medskip\label{f6_1b3}\\
& \text{ \ \ \ \ \ \ }=\frac{a_{11}(b_{2}-x)+a_{22}(x-b_{1})}{b_{2}-b_{1}%
}\,+\,g(v)\nonumber
\end{align}
with%
\begin{equation}
g(v):=\frac{2\,x-(b_{2}+b_{1})}{b_{2}-b_{1}}\,v\,+\,\sqrt{|a_{12}|^{2}+v^{2}%
}.\label{f6_1b4}%
\end{equation}
Notice that if $x$ moves between $b_{1}$ and $b_{2}$, then the first term in
the second line of (\ref{f6_1b3}) moves between $a_{11}$ and $a_{22}$, and the
coefficient in front of $v$ in the second term moves between $-1$ and $1$. The
assumption made after (\ref{f6_1a4}) with respect to the square root
transforms into $\sqrt{|a_{12}|^{2}+v^{2}}\approx v$ for $v$\ near $\infty$.
It is evident that $g$ is analytic and single-valued throughout $\overline
{\mathbb{C}}\diagdown\lbrack-i\,|a_{12}|,\,\allowbreak i\,|a_{12}|]$. From
(\ref{f6_1b3}) and (\ref{f6_1a1}) we deduce the representation%
\begin{equation}
w_{A,B}(x)=\frac{2}{b_{2}-b_{1}}\exp\left(  \frac{a_{11}(b_{2}-x)+a_{22}%
(x-b_{1})}{b_{2}-b_{1}}\right)  \frac{1}{2\pi i}\oint\nolimits_{C_{1}}%
e^{g(v)}dv,\label{f6_1c1}%
\end{equation}
where again $C_{1}$ is a positively oriented Jordan curve, which is contained
in the ring domain $\mathbb{C}\diagdown\left[  -i\,|a_{12}|,\,i\,|a_{12}%
|\right]  $. Shrinking this curve to the interval $[-i\,|a_{12}|,\,\allowbreak
i\,|a_{12}|]$ yields that%
\begin{align}
& w_{A,B}(x)\,=\,\frac{1}{(b_{2}-b_{1})\,\pi}\exp\left(  \frac{a_{11}%
(b_{2}-x)+a_{22}(x-b_{1})}{b_{2}-b_{1}}\right)  \times\medskip\label{f6_1c2}\\
& \text{ \ \ \ \ \ \ \ \ \ }\times\int_{-|a_{12}|}^{|a_{12}|}\exp\left(
-i\,\frac{b_{2}+b_{1}-2\,x}{b_{2}-b_{1}}\,v\right)  \left[  e^{\sqrt
{|a_{12}|^{2}-v^{2}}}-e^{-\sqrt{|a_{12}|^{2}-v^{2}}}\right]
dv,\ \ \ \nonumber
\end{align}
and further that%
\begin{align}
& w_{A,B}(x)\,=\,\frac{4}{(b_{2}-b_{1})\,\pi}\exp\left(  \frac{a_{11}%
(b_{2}-x)+a_{22}(x-b_{1})}{b_{2}-b_{1}}\right)  \times\medskip\label{f6_1c3}\\
& \text{ \ \ \ \ \ \ \ \ \ \ \ \ \ \ \ \ \ \ \ \ }\times\int_{0}^{|a_{12}%
|}\cos\left(  \frac{b_{2}+b_{1}-2\,x}{b_{2}-b_{1}}\,v\right)  \,\sinh\left(
\sqrt{|a_{12}|^{2}-v^{2}}\right)  dv,\ \ \ \nonumber
\end{align}
which proves formula (\ref{f1_1e}).

\subsection{\label{sec6_2}The Positivity of $w_{A,B}$}

\qquad Since Proposition \ref{sec1_prop1} is a special case of Theorem
\ref{sec1_thm2}, and since the matrices $A$ and $B$ have been given in the
special form of Assumption 3 in Subsection \ref{sec5_1}, the positivity of
$w_{A,B}(x)$ for $b_{1}<x<b_{2}$ has in principle already been proved by
Proposition \ref{sec5_prop1} in Subsection \ref{sec5_2}. However, the
prominence of the positivity problem in the BMV conjecture may justify an ad
hoc proof for the special case of dimension $n=2$, which is simpler than the
general approach in Section \ref{sec5}, and may also serve as an illustration
for the basic ideas in this approach.\smallskip

From (\ref{f6_1a1}), (\ref{f6_1b3}), (\ref{f6_1b4}), and (\ref{f6_1c1}), it
follows that we have only to prove that%
\begin{align}
& I_{0}\,:=\,\frac{1}{2\pi i}\oint\nolimits_{C_{1}}e^{g(\zeta)}d\zeta
\medskip\nonumber\\
& \text{ \ \ }=\,\frac{2}{\pi}\int_{0}^{a}\cos\left(  b\,v\right)
\,\sinh\left(  \sqrt{a^{2}-v^{2}}\right)  dv\,>\,0\label{f6_2a1}%
\end{align}
with the function $g$ defined in (\ref{f6_1b4}), $a$ and $b$ abbreviations for%
\begin{equation}
a\,:=\,|a_{12}|\text{ \ \ and \ \ }b\,:=\,b(x)\,=\,\frac{2\,x-(b_{2}+b_{1}%
)}{b_{2}-b_{1}},\text{ \ \ respectively,}\label{f6_2a2}%
\end{equation}
and $C_{1}$ a positively oriented integration path in the ring domain
$\mathbb{C}\diagdown\left[  -i\,a,\,i\,a\right]  $.

Obviously, we have $-1<b(x)<1$ for $b_{1}<x<b_{2}$. The value $I_{0}$ of the
integral in the second line of (\ref{f6_2a1}) depends evenly on the parameter
$b$, and $I_{0}$ is obviously positive for $b=0$. Consequently, we can,
without loss of generality, restrict our investigation to values of
$x\in\left(  b_{1},b_{2}\right)  $ that correspond to values $b\in\left(
-1,0\right)  $, and they are $b_{1}\,<\,x\,<\,(b_{1}+b_{2})/2$.

For a fixed value $x\in\left(  b_{1},(b_{1}+b_{2})/2\right)  $ we now study
the behavior of the function $g$ of (\ref{f6_1b4}) in $\mathbb{C}%
\diagdown\left[  -i\,a,\,i\,a\right]  $. Because of the convention with
respect to the sign of the square root in (\ref{f6_1b4}), we have%
\begin{equation}
g(z)\,\approx\,(1+b)\,z\ \text{\ \ \ for \ \ }z\text{ }\approx\text{ }%
\infty.\label{f6_2a3}%
\end{equation}
The function $\operatorname*{Im}g$ is continuous in $\mathbb{C}$, harmonic in
$\mathbb{C}\diagdown\left[  -i\,a,\,i\,a\right]  $, we have
$\operatorname*{Im}g(\overline{z})=-\operatorname*{Im}g(z)$ for $z\in
\mathbb{C}$, and%
\begin{equation}
\operatorname*{Im}g(z)=\,b\,\operatorname*{Im}(z)\text{ }\left\{
\begin{array}
[c]{cc}%
\,<\,0 & \text{ \ \ for \ \ }z\in(0,\,i\,a]\medskip\\
\,>\,0 & \text{ \ \ for \ \ }z\in\lbrack-i\,a,0).
\end{array}
\right. \label{f6_2a4}%
\end{equation}
From (\ref{f6_2a3}), (\ref{f6_2a4}), $1+b>0$, and the harmonicity of
$\operatorname*{Im}g$, we deduce that the set%
\begin{equation}
\{\,z\,|\,\operatorname*{Im}g(z)=0\,\}\,=\,\mathbb{R}\cup\gamma\label{f6_2b2}%
\end{equation}
implicitly defines an analytic Jordan curve $\gamma$, which is contained in
$\mathbb{C}\diagdown\left[  -i\,a,\,i\,a\right]  $. We parameterize this curve
by $\gamma:\left[  0,2\pi\right]  \longrightarrow\mathbb{C}$ in such a way
that it is positively oriented in $\mathbb{C}$ and that%
\begin{equation}
\gamma|_{(0,\pi)}\subset\{\,\operatorname*{Im}(z)>0\,\}\text{, }%
\gamma(0)=:r_{0}>0,\text{ and }\gamma\left(  2\pi-t\right)  =\overline
{\gamma\left(  t\right)  }\text{ \ for \ }t\in\left[  0,\pi\right]
.\label{f6_2b1}%
\end{equation}
From (\ref{f6_2b2}) it follows that $g$ is real on $\gamma$. Further, we have%
\begin{equation}
\left(  g\circ\gamma\right)  ^{\prime}\left(  t\right)  \,<\,0\text{ \ \ for
\ \ }t\in(0,\pi).\label{f6_2b3}%
\end{equation}
Indeed, if we set $D_{+}:=\operatorname*{Ext}(\gamma)\cap
\{\,\operatorname*{Im}(z)>0\,\}$ and $D_{-}:=\operatorname*{Int}(\gamma
)\cap\{\,\operatorname*{Im}(z)>0\,\}$, then it follows from (\ref{f6_2a3}),
$1+b>0$, (\ref{f6_2a4}), and (\ref{f6_2b2}) that
\[
\operatorname*{Im}g(z)\left\{
\begin{array}
[c]{cc}%
>0\, & \text{ \ \ for \ \ }z\in D_{+}\\
<0 & \text{ \ \ for \ \ }z\in D_{-},
\end{array}
\right.
\]
and with the harmonicity of $\operatorname*{Im}g$ we deduce that%
\[
(\frac{\partial}{\partial n}\operatorname*{Im}g)\circ\gamma\left(  t\right)
\,<\,0\text{ \ \ for \ \ }t\in(0,\pi),
\]
where $\partial/\partial n$ denotes the normal derivative on $\gamma$ pointing
into $D_{-}$. The inequality in (\ref{f6_2b3}) then follows by the
Cauchy-Riemann differential equations and the fact that $g\circ\gamma
=\operatorname{Re}g\circ\gamma$.

With the Jordan curve $\gamma$ and the inequality in (\ref{f6_2b3}) we are
prepared to prove the positivity of the integral $I_{0}$ in (\ref{f6_2a1}).
Using $\gamma$ as integration path in the integral in the first line of
(\ref{f6_2a1}) yields that%
\begin{align}
I_{0}  & =\frac{1}{2\pi i}\int_{0}^{2\pi}e^{g\circ\gamma\left(  t\right)
}\gamma^{\prime}\left(  t\right)  dt\,\,=\,\,\frac{1}{\pi}\operatorname*{Im}%
\int_{0}^{\pi}e^{g\circ\gamma\left(  t\right)  }\gamma^{\prime}\left(
t\right)  dt\medskip\nonumber\\
& =\frac{1}{\pi}\operatorname*{Im}\left[  e^{g\circ\gamma\left(  t\right)
}\gamma\left(  t\right)  \right]  _{0}^{\pi}-\frac{1}{\pi}\operatorname*{Im}%
\int_{0}^{\pi}\left(  g\circ\gamma\right)  ^{\prime}\left(  t\right)
e^{g\circ\gamma\left(  t\right)  }\gamma\left(  t\right)  dt\medskip
\label{f6_2b4}\\
& =-\frac{1}{\pi}\int_{0}^{\pi}\left(  g\circ\gamma\right)  ^{\prime}\left(
t\right)  e^{g\circ\gamma\left(  t\right)  }\operatorname*{Im}\left(
\gamma\left(  t\right)  \right)  dt\,>\,0.\nonumber
\end{align}
Indeed, the second equality in the first line of (\ref{f6_2b4}) is a
consequence of the symmetry relations $\left(  g\circ\gamma\right)  \left(
t\right)  =$ $\left(  g\circ\gamma\right)  \left(  2\pi-t\right)  $,
$\gamma^{\prime}\left(  t\right)  =$ $-\,\overline{\gamma^{\prime}\left(
2\pi-t\right)  }$, and $\gamma\left(  t\right)  =$ $\overline{\gamma\left(
2\pi-t\right)  }$ for $t\in\lbrack0,2\pi)$. The next equality follows from
partial integration, and the last equality is a consequence of
$\operatorname*{Im}\gamma(0)=\operatorname*{Im}\gamma\left(  \pi\right)
=0$\ and $\operatorname*{Im}\left(  g\circ\gamma\right)  \left(  t\right)  =0$
for $t\in\lbrack0,2\pi)$. Finally, the inequality in (\ref{f6_2b4}) is a
consequence of (\ref{f6_2b3}) together with $\operatorname*{Im}\gamma\left(
t\right)  >0$ for $t\in(0,\pi)$.

With (\ref{f6_2b4}) we have verified that $w_{A,B}(x)\,>\,0$ for all
$x\in\left(  b_{1},b_{2}\right)  $, which completes the proof of Proposition
\ref{sec1_prop1}.$\medskip$

\subsection{\label{sec6_3}A Comparison with the Solution in
\cite{MehtaKumar76}}

\qquad In \cite[Formulae (2.13) - (2.16)]{MehtaKumar76} an explicit
representation for the measure $\mu_{A,B}$ has been proved for the case of
dimension $n=2$, in which the expression of the density function $w_{A,B}$
differs considerably in its appearance from representation (\ref{f1_1e}) in
Proposition \ref{sec1_prop1}; it reads\footnote{Formula (2.15) of
\cite{MehtaKumar76}, which is reproduced here as (\ref{f6_3a2}), contains a
misprint; there is written erroneously $2n+1$ instead of $2n-1$ in the
exponent of the denominator. The correction can easily be verified by
following its derivation starting from (2.11) in \cite{MehtaKumar76}.} as%
\begin{align}
& w_{A,B}(x)\,=\,\exp\left(  \frac{a_{11}(b_{2}-x)+a_{22}(x-b_{1})}%
{b_{2}-b_{1}}\right)  \,G_{12}(x)\text{ \ \ with}\medskip\label{f6_3a1}\\
& G_{12}(x)\,=\,\sum_{j=1}^{\infty}\frac{|a_{12}|^{2\,j}}{j!\,(j-1)!}%
\frac{(b_{2}-x)^{n-1}(x-b_{1})^{n-1}}{(b_{2}-b_{1})^{2n-1}},\text{ \ \ }%
b_{1}<x<b_{2}\text{, \ \ \ }\label{f6_3a2}%
\end{align}
where we use the terminology from Proposition \ref{sec1_prop1}. The
representations (\ref{f6_3a2}) and (\ref{f1_1e}) have not only a rather
different appearance, they have also been obtained by very different
approaches. However, they are identical, as will be shown in the next lines.
We have to show that%
\begin{equation}
G_{12}(x)=\frac{4}{(b_{2}-b_{1})\,\pi}\int_{0}^{|a_{12}|}\cos\left(
\frac{b_{2}+b_{1}-2\,x}{b_{2}-b_{1}}\,u\right)  \,\sinh\left(  \sqrt
{|a_{12}|^{2}-u^{2}}\right)  du\label{f6_3a3}%
\end{equation}
for $b_{1}<x<b_{2}$.

We use the same abbreviations $a$\ and $b$ as in (\ref{f6_2a2}). From%
\[
\cos\left(  b\,u\right)  \,\sinh\left(  \sqrt{a^{2}-u^{2}}\right)
\,=\,\sum_{j=0}^{\infty}\sum_{k=1}^{\infty}\frac{(-1)^{j}\,b^{2j}%
}{(2j)!\,(2k-1)!}\frac{u^{2j}(a^{2}-u^{2})^{k}}{\sqrt{a^{2}-u^{2}}}%
\]
and%
\begin{align*}
\int_{0}^{a}\frac{u^{2j}(a^{2}-u^{2})^{k}}{\sqrt{a^{2}-u^{2}}}du  &
=a^{2\,(j+k)}\frac{\Gamma(j+\frac{1}{2})\Gamma(k+\frac{1}{2})}{(j+k)!}%
\medskip\\
& =\pi\,a^{2(j+k)}\frac{(2j)!\,(2k)!}{2^{2(j+k)}\,(j+k)!\,j!\,k!}%
\end{align*}
we deduce that%
\begin{align}
& \int_{0}^{a}\cos\left(  b\,u\right)  \,\sinh\left(  \sqrt{a^{2}-u^{2}%
}\right)  du=\medskip\nonumber\\
& \text{ \ \ \ \ \ \ \ \ \ \ \ \ \ \ \ \ \ \ }=\,\pi\,\sum_{j=0}^{\infty}%
\sum_{k=1}^{\infty}(-1)^{j}\,b^{2j}\,a^{2\,(j+k)}\frac{4^{-(j+k)}%
}{(j+k)!\,j!\,(k-1)!}\medskip\nonumber\\
& \text{ \ \ \ \ \ \ \ \ \ \ \ \ \ \ \ \ \ \ }=\,\pi\,\sum_{n=1}^{\infty}%
\frac{a^{2n}}{4^{n}\,n!\,(n-1)!}\sum_{j=0}^{n-1}(-1)^{j}\frac{(n-1)!}%
{j!\,(n-j-1)!}b^{2\,j}\text{ \ \ \ \ \ }\medskip\label{f6_3a4}%
\end{align}%
\begin{align}
& \text{ \ \ \ \ \ \ \ \ \ \ \ \ \ \ \ \ \ \ }=\,\frac{\pi}{4}\,\sum
_{n=1}^{\infty}\frac{a^{2\,n}}{n!\,(n-1)!}\left(  \frac{1-b^{2}}{4}\right)
^{n-1}\medskip\nonumber\\
& \text{ \ \ \ \ \ \ \ \ \ \ \ \ \ \ \ \ \ \ }=\,\frac{\pi}{4}\,\sum
_{n=1}^{\infty}\frac{|a_{12}|^{2\,n}}{n!\,(n-1)!}\frac{(b_{2}-x)^{n-1}%
(x-b_{1})^{n-1}}{(b_{2}-b_{1})^{2\left(  n-1\right)  }}.\text{
\ \ \ \ \ \ \ \ \ \ \ }\nonumber
\end{align}
The last equality in (\ref{f6_3a4}) follows from%
\[
\frac{1}{4}(1-b^{2})=\frac{1}{4}\left(  1-\left(  \frac{b_{2}+b_{1}%
-2\,x}{b_{2}-b_{1}}\right)  ^{2}\right)  =\frac{(b_{2}-x)(x-b_{1})}%
{(b_{2}-b_{1})^{2}}.
\]
With (\ref{f6_3a4}) identity (\ref{f6_3a3}) is proved.\bigskip

\textit{Acknowledgements}. The author is very grateful to \textit{Pierre
Moussa} for his enthusiastic welcome of the first version of the present paper
and for valuable hints to interesting earlier publications, to \textit{Peter
Landweber} for many improvements in the manuscript, and last but not least, to
\textit{Elliott H. Lieb} and \textit{Robert Seiringer}\ for a valuable
discussion, and especially for the suggestion to simplify and shorten the
proof of the conjecture dramatically by going for a direct verification of the
formulae (\ref{f1_3a}), (\ref{f1_3b}), and (\ref{f1_3c}) in Theorem
\ref{sec1_thm2}, as is now done in Section \ref{sec3}. In the original version
of the paper these formulae were proved by a lengthy, asymptotic analysis of
the function (\ref{f1_1a0}) with a subsequent use of the Post-Widder inversion
formulae for Laplace transforms.

\bibliographystyle{plain}

\begin{thebibliography}{10}

\bibitem{BMV75}
D.~Bessis, P.~Moussa, and M.~Villani.
\newblock {Monotonic converging variational approximations to the functional
  integrals in quantum statistical mechanics.}
\newblock {\em J. Math. Phys.}, 16:2318--2325, 1975.

\bibitem{Burgdorf11}
Sabine Burgdorf.
\newblock Sums of hermitian squares as an approach to the {BMV} conjecture.
\newblock {\em Linear and Multilinear Algebra}, 59:1--9, 2011.

\bibitem{CollinsDykemaTorres10}
Beno{\^i}t Collins, Kenneth~J. Dykema, and Francisco Torres-Ayala.
\newblock {Sum-of-squares results for polynomials related to the
  Bessis-Moussa-Villani conjecture.}
\newblock {\em J. Stat. Phys.}, 139(5):779--799, 2010.

\bibitem{Donoghue}
William~F. Donoghue.
\newblock {\em Monotone matrix functions and analytic continuation}.
\newblock Spring{\-}er-Ver{\-}lag, Berlin, 1974.

\bibitem{FarkasKra92}
H.M. Farkas and I.~Kra.
\newblock {\em Riemann surfaces}.
\newblock Graduate texts in mathematics. Springer-Verlag, 1992.

\bibitem{FleischhackFriedland10}
Christian Fleischhack and Shmuel Friedland.
\newblock {Asymptotic positivity of Hurwitz product traces: two proofs.}
\newblock {\em Linear Algebra Appl.}, 432(6):1363--1383, 2010.

\bibitem{Grafendorfer07}
G.~Grafendorfer.
\newblock Hardening of the {BMV} conjecture.
\newblock Technical report, Wirtschafts - Mathematik, Technische Universit\"{a}t
  Wien, 2007.
\newblock Available on http://www.math.ethz.ch/~ggeorg/files/thesis.pdf.

\bibitem{Hansen06}
Frank Hansen.
\newblock {Trace functions as Laplace transforms.}
\newblock {\em J. Math. Phys.}, 47(4):043504, 11 p., 2006.

\bibitem{Hillar07}
Christopher~J. Hillar.
\newblock {Advances on the Bessis-Moussa-Villani trace conjecture.}
\newblock {\em Linear Algebra Appl.}, 426(1):130--142, 2007.

\bibitem{HillarJohnson05}
Christopher~J. Hillar and Charles~R. Johnson.
\newblock {On the positivity of the coefficients of a certain polynomial
  defined by two positive definite matrices.}
\newblock {\em J. Stat. Phys.}, 118(3-4):781--789, 2005.

\bibitem{Haegele07}
Daniel H\"{a}gele.
\newblock {Proof of the cases $p \leq 7$ of the Lieb-Seiringer formulation of
  the Bessis-Moussa-Villani conjecture.}
\newblock {\em J. Stat. Phys.}, 127(6):1167--1171, 2007.

\bibitem{JohnsonLeichenauerMcNamaraCostas05}
Charles~R. Johnson, Stefan Leichenauer, Peter McNamara, and Roberto Costas.
\newblock {Principal minor sums of $(A + tB)^m$.}
\newblock {\em Linear Algebra Appl.}, 411:386--389, 2005.

\bibitem{KlepSchweighofer08}
Igor Klep and Markus Schweighofer.
\newblock {Sums of Hermitian squares and the BMV conjecture.}
\newblock {\em J. Stat. Phys.}, 133(4):739--760, 2008.

\bibitem{LandweberSpeer09}
Peter~S. Landweber and Eugene~R. Speer.
\newblock {On D. H\"agele's approach to the Bessis-Moussa-Villani conjecture.}
\newblock {\em Linear Algebra Appl.}, 431(8):1317--1324, 2009.

\bibitem{LiebSeiringer04}
Elliott~H. Lieb and Robert Seiringer.
\newblock {Equivalent forms of the Bessis-Moussa-Villani conjecture.}
\newblock {\em J. Stat. Phys.}, 115(1-2):185--190, 2004.

\bibitem{LiebSeiringer12}
Elliott~H. Lieb and Robert Seiringer.
\newblock {Further implications of the Bessis-Moussa-Villani conjecture}.
\newblock 2012.
\newblock Available under: arXiv:1206.0460v1 [math-ph].

\bibitem{MehtaKumar76}
M.~L. Mehta and K.~Kumar.
\newblock {On an integral representation of the function $Tr[exp(A-\lambda
  B)]$}.
\newblock {\em J. Phys. A}, 9:197--206, 1976.

\bibitem{Moussa00}
Pierre Moussa.
\newblock {On the representation of $\text{Tr}(e^{(A-\lambda B)})$ as a Laplace
  transform.}
\newblock {\em Rev. Math. Phys.}, 12(4):621--655, 2000.

\bibitem{Stahl11}
H.~Stahl.
\newblock Proof of the {BMV} conjecture, 2011.
\newblock posted under arXiv:1107.4875 [math.CV].

\bibitem{Widder}
David~V. Widder.
\newblock {\em The Laplace transform}.
\newblock Princeton University Press, 1946.

\end{thebibliography}

\newif\ifabfull\abfulltrue

\bigskip
\end{document}